\documentclass[12pt]{amsart}
\usepackage{amsmath}
\usepackage{geometry,amsthm,graphics,tabularx,amssymb,shapepar}
\usepackage{amscd}
\usepackage[all,2cell,dvips]{xy}

\usepackage{latexsym,amsfonts,amsxtra}
\usepackage[mathscr]{eucal}
\newcommand{\BA}{{\mathbb {A}}}

\newcommand{\BD}{{\mathbb {D}}}

\newcommand{\Bk}{{\mathbf {k}}}

\newcommand{\BFD}{{\mathbf {D}}}

\newcommand{\CFH}{\mathcal {FH}}
\newcommand{\CDH}{\mathcal {DH}}

\newcommand{\CP}{{\mathcal {P}}}

\newcommand{\GL}{{\mathrm{GL}}}

\newcommand{\Hom}{{\mathrm{Hom}}}

\newcommand{\Lie}{{\mathrm{Lie}}}

\newcommand{\rank}{{\mathrm{rank}}}

\newcommand{\tr}{{\mathrm{tr}}}

\newcommand{\cover}[1]{\widetilde{#1}}

\newcommand{\vsp}{{\vspace{0.2in}}}

\newcommand{\con}{\textit{C}}

\newcommand{\Diff}{\operatorname{DO}}
\newcommand{\diag}{\operatorname{diag}}

\newcommand{\oleft}{\operatorname{left}}
\newcommand{\oright}{\operatorname{right}}
\newcommand{\open}{\operatorname{open}}

\newcommand{\oM}{\operatorname{M}}
\newcommand{\oO}{\operatorname{O}}
\newcommand{\oSp}{\operatorname{Sp}}
\newcommand{\oSO}{\operatorname{SO}}
\newcommand{\oS}{\operatorname{S}}
\newcommand{\oR}{\operatorname{R}}
\newcommand{\oT}{\operatorname{T}}
\newcommand{\oU}{\operatorname{U}}
\newcommand{\oZ}{\operatorname{Z}}
\newcommand{\oD}{\textit{D}}

\newcommand{\g}{\mathfrak g}
\newcommand{\gC}{{\mathfrak g}_{\C}}

\renewcommand{\a}{\mathfrak a}

\newcommand{\n}{\mathfrak n}

\renewcommand{\l}{\mathfrak l}

\newcommand{\s}{\mathfrak s}

\renewcommand{\sl}{\mathfrak s \mathfrak l}
\newcommand{\gl}{\mathfrak g \mathfrak l}
\newcommand{\fM}{\mathfrak M}
\newcommand{\fZ}{\mathfrak Z}

\newcommand{\Z}{\mathbb{Z}}
\newcommand{\C}{\mathbb{C}}
\newcommand{\R}{\mathbb R}
\renewcommand{\H}{\mathbb{H}}
\newcommand{\K}{\mathbb{K}}
\renewcommand{\S}{\mathbf S}

\newcommand{\la}{\langle}
\newcommand{\ra}{\rangle}

\newcommand{\be}{\begin {equation}}
\newcommand{\ee}{\end {equation}}
\newcommand{\bee}{\begin {equation*}}
\newcommand{\eee}{\end {equation*}}

\def\bks{{\backslash}}

\def\diag{{\rm diag}}

\theoremstyle{Theorem}

\newtheorem{prp}{Proposition}[section]

\newtheorem{lemp}[prp]{Lemma}

\newtheorem{prpp}[prp]{Proposition}

\newtheorem{coj}{Conjecture}[section]

\newtheorem{thmc}[coj]{Theorem}

\newtheorem{prpc}[coj]{Proposition}
\theoremstyle{Plain}
\newtheorem{rmk}[coj]{Remark}

\theoremstyle{Plain}

\theoremstyle{Theorem}
\newtheorem{dfn}{Definition}[section]

\newtheorem{lemd}[dfn]{Lemma}
\newtheorem{sublemd}[dfn]{Sublemma}
\newtheorem{dfnd}[dfn]{Definition}

\begin{document}
\renewcommand{\theequation}{\arabic{equation}}
\numberwithin{equation}{section}

\theoremstyle{Theorem}

\newtheorem{dtt}{Definition}[section]
\newtheorem{tvanishing}[dtt]{Lemma}
\newtheorem{tvanishing2}[dtt]{Lemma}

\newtheorem{conj1.1}{Conjecture}[section]
\newtheorem{conj1.2}[conj1.1]{Conjecture}
\newtheorem{conj1.3}[conj1.1]{Conjecture}
\newtheorem{first-occur}[conj1.1]{Proposition}
\newtheorem{dum}[conj1.1]{Definition}
\newtheorem{UchiM}[conj1.1]{Proposition}

\newtheorem{localization}{Lemma}[section]
\newtheorem{dslice}[localization]{Definition}
\newtheorem{localization0}[localization]{Lemma}
\newtheorem{localization1}[localization]{Lemma}
\newtheorem{localization2}[localization]{Lemma}

\newtheorem{dmp}{Definition}[section]
\newtheorem{tvanishing3}[dmp]{Lemma}
\newtheorem{duc}[dmp]{Definition}
\newtheorem{localization3}[dmp]{Lemma}
\newtheorem{localization4}[dmp]{Sublemma}
\newtheorem{localization5}[dmp]{Lemma}

\newtheorem{m2}{Proposition}[section]
\newtheorem{m23}[m2]{Lemma}

\newtheorem{m3}{Proposition}[section]
\newtheorem{m32}[m3]{Proposition}
\newtheorem{m33}[m3]{Lemma}
\newtheorem{m34}[m3]{Lemma}
\newtheorem{gl3p}[m3]{Lemma}
\newtheorem{gl3z1}[m3]{Lemma}
\newtheorem{gl3z0}[m3]{Lemma}
\newtheorem{gl3w1}[m3]{Lemma}
\newtheorem{gl3w0}[m3]{Lemma}
\newtheorem{gl3g0}[m3]{Lemma}
\newtheorem{nz1}[m3]{Lemma}
\newtheorem{gl3g0w1}[m3]{Lemma}
\newtheorem{gl3g1}[m3]{Lemma}
\newtheorem{fourier}[m3]{Lemma}

\newtheorem{m4}{Proposition}[section]
\newtheorem{m4slice}[m4]{Lemma}
\newtheorem{m4stable}[m4]{Lemma}
\newtheorem{m42}[m4]{Lemma}
\newtheorem{m43}[m4]{Lemma}
\newtheorem{m44}[m4]{Lemma}
\newtheorem{m45}[m4]{Lemma}
\newtheorem{m46}[m4]{Lemma}
\newtheorem{h6z4}[m4]{Proposition}
\newtheorem{h6z4l}[m4]{Lemma}

\newtheorem{m6}{Proposition}[section]
\newtheorem{m62}[m6]{Lemma}
\newtheorem{m63}[m6]{Lemma}
\newtheorem{m64}[m6]{Lemma}
\newtheorem{m65}[m6]{Lemma}
\newtheorem{m66}[m6]{Lemma}
\newtheorem{m67}[m6]{Lemma}

\newtheorem{smallm}{Proposition}[section]
\newtheorem{smallm0}[smallm]{Lemma}
\newtheorem{smallm1}[smallm]{Lemma}
\newtheorem{smallm2}[smallm]{Lemma}
\newtheorem{smallm3}[smallm]{Lemma}
\newtheorem{smallm4}[smallm]{Lemma}
\newtheorem{mp4}[smallm]{Lemma}
\newtheorem{mp5}[smallm]{Proposition}

\newtheorem{mp1}{Lemma}[section]
\newtheorem{mp2}[mp1]{Lemma}
\newtheorem{mp3}[mp1]{Lemma}
\newtheorem{mp6}[mp1]{Lemma}
\newtheorem{mp7}[mp1]{Lemma}
\newtheorem{um-transitivity}[mp1]{Lemma}
\newtheorem{transitivity}[mp1]{Proposition}
\newtheorem{symmetry}[mp1]{Proposition}

\newtheorem{gelfand}{Proposition}[section]

\newtheorem{quat}{Theorem}[section]
\newtheorem{quat1}[quat]{Theorem}
\newtheorem{quat2}[quat]{Theorem}
\newtheorem{quat3}[quat]{Theorem}
\newtheorem{quat4}[quat]{Theorem}

\newtheorem{Whittaker}{Theorem}[section]
\newtheorem{Whittaker2}[Whittaker]{Theorem}
\newtheorem{Whittaker3}[Whittaker]{Theorem}
\newtheorem{Whittaker4}[Whittaker]{Theorem}

\title[Ginzburg-Rallis models]{Uniqueness of Ginzburg-Rallis models:
the Archimedean case}

\author [D. Jiang] {Dihua Jiang}
\address{School of Mathematics\\
University of Minnesota\\
206 Church St. S.E., Minneapolis\\
MN 55455, USA} \email{dhjiang@math.umn.edu}

\author [B. Sun] {Binyong Sun}
\address{Academy of Mathematics and Systems Science\\
Chinese Academy of Sciences\\
Beijing, 100190,  P.R. China} \email{sun@math.ac.cn}

\author [C.-B. Zhu] {Chen-Bo Zhu}
\address{Department of Mathematics\\
National University of Singapore\\
2 Science drive 2\\
Singapore 117543} \email{matzhucb@nus.edu.sg}



\begin{abstract}
In this paper, we prove the uniqueness of Ginzburg-Rallis models
in the archimedean case. As a key ingredient, we introduce a new
descent argument based on two geometric notions attached to
submanifolds, which we call metrical properness and unipotent
$\chi$-incompatibility.
\end{abstract}

\maketitle


\section{Introduction and main results}
\label{intro}

In year 2000, Ginzburg and Rallis formulated a conjecture to
characterize the nonvanishing of central values of partial exterior
cube L-functions attached to irreducible cuspidal automorphic
representations of $\GL_6$ in terms of certain periods
(\cite{GR00}). This is analogous to the Jacquet conjecture for the
triple product L-functions for $\GL_2$ (established in full by
Harris and Kudla in \cite{HK04}), and to the Gross-Prasad conjecture
for classical groups (\cite{GP92, GP94, GJR04, GJR05,GJR09}).

To be precise, let $\BA$ be the ring of adeles of a number field
$\Bk$. Fix a nontrivial unitary character $\psi_\BA$ of
$\Bk\bks\BA$, and a (non-necessarily unitary) character
$\chi_{\BA^\times}$ of $\Bk^\times \bks\BA^\times$. For any
quaternion algebra $\BFD$ over $\Bk$, denote
$G_{\BFD}=\GL_3(\BFD)$, and $S_\BFD$ its subgroup consisting of
elements of the form
\begin{equation}
           \left[
            \begin{array}{ccc}
              a&b&d\\ 0&a&c\\0 & 0&a
              \end{array}
           \right].
\end{equation}
Define a character $\chi_{S_\BFD}$ of $S_\BFD(\BA)$ by
\begin{equation}
   \chi_{S_\BFD}\left(
   \left[
            \begin{array}{ccc}
              1&b&d\\ 0&1&c\\0 & 0&1
              \end{array}
           \right]\cdot
\left[
            \begin{array}{ccc}
              a&0&0\\ 0&a&0\\0 & 0&a
              \end{array}
           \right]
\right)=\chi_{\BA^\times}(\det(a))\,\psi_{\BA}(\tr(b+c)),
\end{equation}
where $\det$ and $\tr$ stand for the reduced norm and the reduced
trace, respectively.

Let $\varphi_\BFD$ be an automorphic form on $G_\BFD(\Bk)\bks
G_\BFD(\BA)$. The Ginzburg-Rallis period
$\CP_{\chi_{S_\BFD}}(\varphi_\BFD)$ of $\varphi_\BFD$ is defined
by the following integral
\begin{equation*}
\CP_{\chi_{S_\BFD}}(\varphi_\BFD)= \int_{\BA^\times
S_\BFD(\Bk)\bks\S_\BFD(\BA)}\varphi_\BFD(s)(\chi_{S_\BFD}(s))^{-1}ds,
\end{equation*}
where $\BA^\times$ is identified with the center of $G_\BFD(\BA)$.
The Ginzburg-Rallis conjecture can then be stated as follows.

\begin{coj}
\label{conj1.1} {\rm (Ginzburg-Rallis, \cite{GR00})} Let $\pi$ be an
irreducible cuspidal automorphic representation of
 $\GL_6(\BA)$ with central character $\chi_{\BA^\times}^2$. For any
quaternion algebra $\BFD$ over $\Bk$, denote by $\pi_\BFD$ the
generalized Jacquet-Langlands correspondence of $\pi$, which is
either zero or an irreducible cuspidal automorphic representation
of $G_\BFD(\BA)$. Consider the irreducible representation
$\Lambda^3\otimes \C^1$ of the L-group $\GL_6(\C)\times
\GL_1(\C)$, where $\Lambda^3$ is the exterior cube product of the
standard representation of $\GL_6(\C)$, and $\C^1$ is the standard
representation of $\GL_1(\C)$. The partial L-function
$L^S(s,\pi\otimes \chi_{\BA^\times}^{-1},\Lambda^3\otimes \C^1)$
does not vanish at $s=\frac{1}{2}$ if and only if there exists a
unique quaternion algebra $\BFD$ such that
\begin{itemize}
\item[(a)] the period $\CP_{\chi_{S_\BFD}}(\varphi_\BFD)$ is
nonzero for some $\varphi_\BFD\in\pi_\BFD$; and \item[(b)] for any
quaternion algebra $\BFD'$ which is not isomorphic to $\BFD$, the
period $\CP_{\chi_{S_{\BFD'}}}(\varphi_{\BFD'})$ is zero for every
$\varphi_{\BFD'}\in\pi_{\BFD'}$.
\end{itemize}
\end{coj}

See \cite{GR00} and \cite{GJ} for some partial results on the
conjecture.

\vsp
 We consider the corresponding local theory. Let $\K$ be a
local field of characteristic zero. Fix a nontrivial unitary
character $\psi_\K$ of $\K$, and an arbitrary character
$\chi_{\K^\times}$ of $\K^\times$. For any quaternion algebra
$\BD$ over $\K$, denote $G_\BD=\GL_3(\BD)$ and define its subgroup
$S_\BD$ as in the number field case. We also define the local
analogy $\chi_{S_\BD}$ of $\chi_{S_\BFD}$, by the same formula in
terms of the characters $\psi_\K$ and $\chi_{\K^\times}$.

If $\K$ is nonarchimedean, we let $V_\BD$ be an irreducible smooth
representation of $G_\BD$, and if $\K$ is archimedean, let $V_\BD$
be an irreducible representation of $G_\BD$ in the class $\CFH$.
The notion of representations in the class $\CFH$ will be
explained in Section \ref{theoremA}.

As in the proof of Jacquet conjecture, in order to tackle the
Ginzburg-Rallis Conjecture, the first basic property that we should
establish is

\begin{coj}
\label{conj1.2} The Ginzburg-Rallis models on $V_\BD$ is unique up
to scalar, i.e.,
$$
  \dim \Hom_{S_\BD} (V_\BD,\C_{\chi_{S_\BD}})\leq 1,
$$
where $\C_{\chi_{S_\BD}}$ is the one dimensional representation of
$S_\BD$ given by the character $\chi_{S_\BD}$.
\end{coj}

This conjecture has been expected since the work \cite{GR00} and was first discussed with details in \cite{J08}. In her
Minnesota thesis (directed by the first named author), Nien proved
Conjecture \ref{conj1.2} in the nonarchimedean case (\cite{Nien}).
We remark that there is a generalization of the Ginzburg-Rallis
models to $\GL _{3n}$, which may be viewed as the ``three block"
version of the Whittaker models for $\GL _n$. As noted in
\cite{Nien}, the local uniqueness property is not expected to hold
for the generalized Ginzburg-Rallis models for $\GL _{3n}$ with
$n>2$.

\vsp The first main purpose of this paper is to prove the
archimedean case of Conjecture \ref{conj1.2}, which requires substantially
more delicate analysis than the nonarchimedean case.

From now on, we will assume that $\K$ is the archimedean local field
$\R$ or $\C$.

\begin{thmc}
\label{thm:mainA} Let $V_\BD$ be an irreducible representation of $G_\BD$ in the class $\CFH$. Then
\[
  \dim \Hom_{S_\BD} (V_\BD,\C_{\chi_{S_\BD}})\leq 1.
\]
\end{thmc}

Note that the notion of representations in the class $\CFH$
includes the requirement of moderate growth. This has the
implication that
\[
 \Hom_{S_\BD} (V_\BD,\C_{\chi_{S_\BD}})=0,
\]
if one replaces the additive character $\psi_\K$ with one which is
not unitary.

Ginzburg-Rallis models are so called ``mixed models", as the group
$S_\BD$ is neither unipotent nor reductive. On the other hand, we
have the Whittaker models and linear models, where the subgroup
involved is unipotent or reductive, respectively. By now we know
that uniqueness of Whittaker models is relatively easy to
establish (see Section \ref{sub-whittaker} for a short proof). The
study of uniqueness of linear models was initiated by
Jacquet-Rallis in \cite{JR96}, and there have been a number of
recent advances in this direction (see \cite{AGRS, AG-multiplicity, SZ}, for
example). We remark that in each case, a good understanding of
algebraic and geometric structure of the orbital decomposition is
required. (The task is made easier by geometric invariant theory,
see \cite{AG2}.) Although in some special cases, one may reduce
uniqueness of mixed models to that of linear models (c.f.
\cite{JR96, AGJ} and Remark \ref{bessel} of this section), there
is still a lack of general techniques to treat the mixed model
problems (save for a few low rank cases; see for example
\cite{BR07}). Besides a proof of Theorem \ref{thm:mainA}, another
main purpose of this paper is to introduce a descent method in the
archimedean case that reduces uniqueness of mixed models to that
of linear models. We carry out the descent process for the
Ginzburg-Rallis model, which is considered as an exceptional
model, and is also sufficiently complicated to reveal difficulties
in general archimedean mixed model problems.

\vsp

We introduce some notations. For any natural number $n$, denote by
$\gl_n(\K)$ the space of $n\times n$ matrices with entries in $\K$.
When the quaternion algebra $\BD$ is split, we fix an identification
of $\BD$ with $\gl_2(\K)$, and then $G_\BD$ is identified with
$\GL_6(\K)$. For a square matrix $x$, if its entries are from $\K$,
denote by $x^\tau$ its transpose. If $\BD$ is not split and $x\in
G_\BD=\GL_3(\BD)$, set
\[
  x^\tau=\textrm{the transpose of $\bar x$},
\]
where ``$\bar{\phantom{x}}$" denotes the (element-wise) quaternionic
conjugation.

Define the real trace form $\la\,,\,\ra_\R$ on the Lie algebra
$\gl_3(\BD)$ of $G_\BD$ by
\begin{equation}\label{rtrace}
    \la x,y\ra_{\R}=\left\{\begin{array}{l}
                        \textrm{the real part of the trace of $xy$,
                        $\,\,$ if $\BD$ is split},\\
                        \textrm{the reduced trace of $xy$, \qquad \qquad othewise}.\\
                        \end{array}
                        \right.
\end{equation}
Denote by $\Delta_\BD$ the Casimir element with respect to
$\la\,,\,\ra_\R$, which is viewed as a bi-invariant differential
operator on $G_\BD$.

We will see in Section \ref{theoremA} that by (a general form of)
the Gelfand-Kazhdan criterion, Theorem \ref{thm:mainA} is implied by
the following

\begin{thmc}
\label{thm:mainB}
Let $f$ be a
tempered generalized function on $G_\BD$, which is an eigenvector of
$\Delta_\BD$. If $f$ satisfies
\begin{equation}\label{fsx}
  f(sx)=f(xs^\tau)=\chi_{S_\BD}(s)f(x), \quad \textrm{for all } s\in S_\BD,
\end{equation}
then \[f(x)=f(x^\tau).\]
\end{thmc}

The notion of tempered generalized functions will be explained in
Section \ref{snash}. We remark that the equalities in the theorem
are to be understood as equalities of generalized functions, and
$f(sx)$ denotes the left translate of $f$ by $s^{-1}$. Similar
notations apply throughout the article.

\vsp

Assume now that $\BD$ is split. Thus $G=\GL_6$. (We drop the
subscript $\BD$, and the coefficient field $\K$ in all notations.)
The non-split case, which is simpler, will be investigated at the
end of Section \ref{theoremB}.

Following a well-known scheme of Bruhat, we first decompose
\[ G=\bigsqcup_{R}G_R
\]
into $P$-$P^\tau$ double cosets, where
\[
  P=\left\{\,\left[\begin{array}{ccc} a_1&b&d\\ 0&a_2&c\\0&0&a_3\\
                   \end{array}\right]\in G\,
 \right\}
\]
is a parabolic subgroup of $G$ containing $S$.

The proof of Theorem \ref{thm:mainB} will consist of three steps and
will involve three types of arguments:

\vsp \noindent (a) the transversality of certain vector fields to
all except four $G_R$'s, among the twenty one $P$-$P^\tau$ double
cosets of $G$. The technique is due to Shalika \cite{Shalika}. This
allows us to focus the attention to the open submanifold $G'$ of $G$
consisting of the four exceptional double cosets.

\vsp \noindent (b) a descent argument based on two new notions
attached to submanifolds, which we call metrical properness
(Definition \ref{dmp}) and unipotent $\chi$-incompatibility
(Definition \ref{duc}), as well a synthesis of these two notions which we call
$\oU_{\chi}\oM$ property (Definition \ref{dum}).
This lies at the heart of our approach and
forms the main part of our argument. It leads us eventually to two
linear model problems: the uniqueness of trilinear models for
$\GL_2$, and the multiplicity one property for the pair
$(\GL_2,\GL_1)$.

\vsp \noindent (c) use of the oscillator representation to
conclude the uniqueness of the two afore-mentioned linear models.

\vsp

For Step (c) which is relatively easy, we just appeal to the following

\begin{prpc}{\rm (\cite[Theorem C.7]{Pr89})}
\label{first-occur} Let $E$ be a finite dimensional non-degenerate
quadratic space over $\K$, and let the orthogonal group $\oO(E)$ act
on $E^k$ diagonally, where $k$ is a positive integer. If $k<\dim E$,
and if a tempered generalized function $f$ on $E^k$ is
$\oSO(E)$-invariant, then $f$ is $\oO(E)$-invariant.
\end{prpc}

The above proposition may also be stated as that the determinant
character of $\oO(E)$ does not occur in Howe duality correspondence
of $(\oO(E), \oSp (2k))$ if $k<\dim E$. In fact the determinant
character occurs if and only if $k\geq \dim E$. See \cite[Theorem
2.2]{LZ97}.

\vsp The descent process reveals a very interesting interplay
between the Ginzburg-Rallis model and other (smaller) models. The
first model occurring is as follows. Take the maximal Levi
subgroup $G_{4,2}=\GL_4\times\GL_2$ of $G$, and write
\begin{equation}
\label{ds42} S_{4,2}=G_{4,2}\cap
S=\left\{\begin{pmatrix}a&b&0\\0&a&0\\0&0&a\end{pmatrix}\mid
a\in\GL_2, \,b\in \gl_2 \right\}.
\end{equation} In the course of proof of Theorem \ref{thm:mainB}, we
find that for any irreducible representation
$\pi$ of $G_{4,2}$ in the class $\CFH$,
\begin{equation}\label{us42}
   \dim \Hom_{S_{4,2}}(\pi,\chi_{S_{4,2}})\leq 1,
\end{equation}
where $\chi_{S_{4,2}}$ is the restriction of the character
$\chi_{S}$ to $S_{4,2}$. A proof of (\ref{us42}) will be given in
Section \ref{sub-gl4}.

\begin{rmk} \label{bessel} This model may be viewed as the Bessel
model for the orthogonal group pair $(\oO_6,\oO_3)$, via the
(incidental) identification of low rank algebraic groups. In the
p-adic case, for a general pair $(\oO_m,\oO_n)$ with $m>n$ and
having different parity (and its analog for unitary groups), Gan,
Gross and Prasad reduce the uniqueness of Bessel models to the
Multiplicity One Theorems proved by Aizenbud, Gourevitch, Rallis
and Schiffmann (\cite{AGRS,GGP09}). In the archimedean case, the
uniqueness of Bessel models for general linear groups, unitary
groups and orthogonal groups was proved by the authors
(\cite{JSZ09}), using a different reduction technique (from the
p-adic case) and the archimedean Multiplicity One Theorems proved
in \cite{SZ}. Note that the latter for general linear groups is
independently due to Aizenbud and Gourevitch
(\cite{AG-multiplicity}).
\end{rmk}

To examine the case $(G_{4,2}, S_{4,2})$, we perform a further
descent. Consider the maximal Levi subgroup
$$
   G_{2,2,2}=\GL_2\times\GL_2\times\GL_2
$$
of $G_{4,2}$  and the intersection
$$
   S_{2,2,2}=G_{2,2,2}\cap S_{4,2}=\GL_2^\Delta,
$$
which embeds diagonally into $G_{2,2,2}$. This is the well-known
case of the trilinear model for $\GL_2$. See Section
\ref{sub-trilinear}.

An interesting phenomenon here is that in order to complete the
proof for the case $(G_{4,2}, S_{4,2})$, one must also consider the
maximal Levi subgroup
$$
G_{3,1,2}=\GL_3\times\GL_1\times\GL_2
$$
of $G_{4,2}=\GL_4\times\GL_2$. This case reduces essentially to
the case $(\GL_3,S_3)$, where
\[ S_3=\left\{
s(c,d,a)=\begin{pmatrix}1&0&d\\0&1&c\\0&0&1\end{pmatrix}
\cdot\begin{pmatrix}a& & \\ &1&\\ & &a\end{pmatrix} | \ c,d \in
\K, a \in \K^{\times} \right\}.
\]
The corresponding character of $S_3$ is given by
$$
\chi_{S_3}(s(c,d,a))=\chi_{\K^{\times}}(a)\psi _{\K}(d).
$$
This is a mixed model. It should come as no surprise that the pair
$(\GL_3,S_3)$ is a special case of the model introduced by
Jacquet-Shalika (\cite{JS90}) to construct the exterior square
L-functions for $\GL_{2n+1}$. The uniqueness for this case was not
known for any $n$. In the course of our proof for Theorem
\ref{thm:mainB}, we shall prove the uniqueness for the pair
$(\GL_3,S_3)$ over archimedean local fields. (The p-adic case
follows similarly.) See Section \ref{sub-gl3}.

\vsp We now describe the contents and the organization of this
paper.
In Section \ref{generality},
we review some generalities on differential operators, generalized
and invariant generalized functions, basics of Nash manifolds and
the associated notion of tempered-ness. In Section \ref{UXM}, we
define the notions of metrical properness, unipotent
$\chi$-incompatibility, and their synthesis, $\oU_{\chi}\oM$
property. Based on these three new notions, we give respectively
three vanishing results on certain spaces of generalized
functions (Lemmas \ref{tvanishing3}, \ref{localization3}, \ref{UchiM}).
In Section \ref{smallsub}, we prove the transversality of
certain vector fields to all but four of the $P$-$P^{\tau}$ double
cosets, which as mentioned allows us to focus our attention to an
open submanifold $G'$ only. In Section \ref{sm4} and Section
\ref{gl6}, we show (through lengthy but straightforward
computations) that a certain submanifold $Z_4$ of $G_{4,2}$, and a
certain submanifold $Z_6$ of $G'$ has $\oU_{\chi}\oM$ properties,
respectively. This eventually reduces our problem to the submanifolds
$\GL_2\times \GL_2$ and $\GL_3\times \GL_1$ of $G$. In Sections
\ref{gl222}  and \ref{sgl3}, we show that certain spaces of
quasi-invariant tempered generalized functions on $\GL_2\times
\GL_2$ and  $\GL_3\times \GL_1$ vanishes.

The complete proof of Theorem \ref{thm:mainB} will be given in
Section \ref{theoremB}. In Section \ref{theoremA}, we derive Theorem
\ref{thm:mainA} from Theorem \ref{thm:mainB}. Finally, in Section
\ref{misc}, we record uniqueness of models occurring in the process
of descent. In addition and as further evidence for the relevance of
the notion of unipotent $\chi$-incompatibility (for mixed models, as
opposed to linear models), we give a quick proof of the uniqueness
of the Whittaker models based on this notion.

\vsp

\noindent Acknowledgements: Dihua Jiang is supported in part by NSF (USA) grant
DMS--0653742 and by a Distinguished Visiting Professorship at the
Academy of Mathematics and System Sciences, the Chinese Academy of
Sciences. Binyong Sun is supported by NUS-MOE grant R-146-000-102-112, and by NSFC grants 10801126 and 10931006.
Chen-Bo Zhu is supported by NUS-MOE grant R-146-000-102-112.

\section{Generalities}
\label{generality}

We emphasize that materials of this section are all known. In
particular nothing is due to the authors.

\subsection{Generalized functions and differential operators}
\label{sub2.1}
 Let $M$ be a smooth manifold.
Denote by $\con_0^\infty(M)$ the space of compactly supported
(complex valued) smooth functions on $M$, which is a complete
locally convex topological vector space under the usual inductive
smooth topology. Denote by $\oD^{-\infty}(M)$ the strong dual of
$\con_0^\infty(M)$, whose members are called distributions on $M$. A
distribution on $M$ is called a smooth density if under local
coordinate, it is the multiple of a smooth function with the
Lebesgue measure. Under the inductive smooth topology, the space
$\oD_0^\infty(M)$ of compactly supported smooth densities is again a
complete locally convex topological vector space, which is
(non-canonically) isomorphic to $\con_0^\infty(M)$. Denote by
$\con^{-\infty}(M)$ the strong dual of $\oD_0^\infty(M)$, whose
members are called generalized functions on $M$. The space
$\con^{\infty}(M)$ of smooth functions in canonically and
continuously embedded in $\con^{-\infty}(M)$, with a dense image.

If $\phi:M\rightarrow M'$ is a smooth map of smooth manifolds,
then the pushing forward sends compactly supported distributions
on $M$ to compactly supported distributions on $M'$. If
furthermore $\phi$ is a submersion, then the pushing forward
induces a continuous linear map
\[
  \phi_*: \oD_0^\infty(M)\rightarrow \oD_0^\infty(M').
\]
We define the pulling back
\begin{equation}
\label{pullbackmap}
  \phi^*: \con^{-\infty}(M')\rightarrow \con^{-\infty}(M)
\end{equation}
as the transpose of $\phi_*$, which extends the usual pulling back
of smooth functions. The map $\phi^*$ is injective if $\phi$ is a
surjective submersion.

\vsp \noindent {\bf Remark}: Pulling back is not canonically defined
for distributions. For this reason, we work with generalized
functions instead of distributions. \vsp

For $k\in \Z$, denote by $\Diff(M)_k$ the Fr\'{e}chet space of
differential operators on $M$ of order at most $k$, which by
convention is $0$ if $k<0$. It is well-known that every differential
operator $D:\con^{\infty}(M)\rightarrow \con^{\infty}(M)$ may be
continuously extended to $D:\con^{-\infty}(M)\rightarrow
\con^{-\infty}(M)$.

Recall that we have the principal symbol map
\[
   \sigma_k: \Diff(M)_k\rightarrow \Gamma^{\infty}(M,
   \oS^k(\oT(M)\otimes_\R  \C)),
\]
where $\oT(M)$ is the real tangent bundle of $M$, $\oS^k$ stands for
the $k$-th symmetric power, and $\Gamma^{\infty}$ stands for smooth
sections. The continuous linear map $\sigma_k$ is specified by the
following rule:
\[\begin{aligned}
\sigma_k(X_1 X_2\cdots X_k)(x)&=X_1(x) X_2(x)\cdots X_k(x), \quad \text{and}\\
&\sigma_k|_{\Diff(M)_{k-1}}=0,
\end{aligned}\]
for all $x\in M$ and all (smooth real) vector fields $X_1,X_2,
\cdots , X_k$ on $M$.

Let $Z$ be a (locally closed) submanifold of $M$. Write
\[
  \operatorname{N}_Z(M)=\oT(M)|_Z/\oT(Z)
\]
for the normal bundle of $Z$ in $M$. Denote by
\[
   \sigma_{k,Z}: \Diff(M)_k\rightarrow \Gamma^{\infty}(Z,
   \oS^k(\operatorname{N}_Z(M)\otimes_\R \C))
\]
the map formed by composing $\sigma_k$ with the restriction map to
$Z$, and followed by the quotient map
\[
   \Gamma^{\infty}(Z, \oS^k(\oT(M)|_Z\otimes_\R  \C))\rightarrow \Gamma^{\infty}(Z,
   \oS^k(\operatorname{N}_Z(M)\otimes_\R \C)).
\]

\begin{dfn}
\label{dtt}
\begin{itemize}
\item[(a)] A vector field $X$ on $M$ is said to be tangential to
$Z$ if $X(z)$ is in the tangent space $\oT_z(Z)$ for all $z\in Z$,
and transversal to $Z$ if $X(z)\notin \oT_z(Z)$ for all $z\in Z$;
more generally \item[(b)] a differential operator $D$ is said to
be tangential to $Z$ if for every point $z\in Z$ there is an open
neighborhood $U_z$ in $M$ such that $D|_{U_z}$ is a finite sum of
differential operators of the form $\varphi X_1 X_2\cdots X_r$,
where $\varphi$ is a smooth function on $U_z$, $r\geq 0$, and
$X_1,X_2,\cdots, X_r$ are vector fields on $U_z$ which are
tangential to $U_z\cap Z$. For $D\in \Diff(M)_k$, it is said to be
transversal to $Z$ if $\sigma_{k,Z}(D)$ does not vanish at any
point of $Z$.
\end{itemize}
\end{dfn}


We introduce some notations. For a locally closed subset $Z$ of
$M$, denote
\begin{equation}
  \con^{-\infty}(M;Z)=\{f\in\con^{-\infty}(U)|\ \text{supp} (f) \subseteq Z\},
\end{equation}
where $U$ is any open subset of $M$ containing $Z$ as a closed
subset. This definition is independent of $U$. For any
differential operator $D$ on $M$, denote
\begin{equation}
   \con^{-\infty}(M;Z; D)=\{f\in \con^{-\infty}(M;Z)|\ Df=0\}.
\end{equation}

We record the following lemma, which is due to Shalika  (c.f. proof
of Proposition 2.10 in \cite{Shalika}).

\begin{lemd}\label{tvanishing2}
Let $D_1$ be a differential operator on $M$ of order $k\geq 1$,
which is transversal to a submanifold $Z$ of $M$. Let $D_2$ be a
differential operator on $M$ which is tangential to $Z$. Then
\[
   \con^{-\infty}(M;Z;D_1+D_2)=0.
\]
\end{lemd}

\subsection{Invariant generalized functions} \label{inv}
Let $H$ be a Lie group, acting smoothly on a manifold $M$. Fix a
character $\chi$ on $H$. Denote by
\begin{equation}
  \con^{-\infty}_\chi(M)= \{f\in \con^{-\infty}(M)|\ f(hx)=\chi(h)f(x), \ \textrm{for }h\in
  H\}
\end{equation}
the space of $\chi $-equivariant generalized functions.

\vsp Let $\fM$ be a submanifold of $M$ and denote
\[
  \rho_{\fM}: H\times \fM\rightarrow M
\]
the action map.

\begin{dfnd}
\label{dslice} \begin{itemize} \item[(a)] We say that $\fM$ is a
local $H$ slice of $M$ if $\rho_{\fM}$ is a submersion, and an $H$
slice of $M$ if $\rho_{\fM}$ is a surjective submersion.
\item[(b)]
Given two submanifolds $\fZ\subset \fM$ of $M$, we say that $\fZ$ is
relatively $H$ stable in $\fM$ if
\[
  \fM\cap H \fZ=\fZ.
\]
\end{itemize}
\end{dfnd}

Note that the relative stable condition amounts to saying that
$H\times \fZ$ is a union of fibres of the action map $\rho_{\fM}$.
We first prove the following two lemmas in a general setting.

\begin{lemd}\label{subsub1}
Let $Z$ be a subset of $\R^m$. Assume that $Z\times \R^n$ is a
submanifold of $\R^{m+n}$. Then $Z$ is a submanifold of $\R^m$.
\end{lemd}
\begin{proof}
Let $z_0\in Z$. Then there is an open neighborhood $U\times V$ of
$(z_0,0)$ in $\R^{m+n}=\R^m\times \R^n$ and a submersion
\[
  \phi: U\times V\rightarrow \R^d
\]
such that
\[
  (U\times V) \cap (Z\times \R^n)= \phi^{-1}(0).
\]

Denote by
\[
  \phi':U\rightarrow \R^d
\]
the restriction of $\phi$ to $U=U\times \{0\}$. Then
\[
  Z\cap U=\phi'^{-1}(0),
\]
and therefore it suffices to show that $\phi'$ is submersive at
every point $z\in Z\cap U$.

Since $\phi$ is submersive, we have that
\[
  d\phi|_{z,0}(\R^m\oplus \R^n)=\R^d.
\]
Since $\phi$ is constant on $\{z\}\times V$, we have that
\[
  d\phi|_{z,0}(\R^n)=0.
\]
Therefore,
\[
  d\phi|_{z,0}(\R^m)=\R^d,
\]
which implies that $\phi'$ is submersive at $z$.
\end{proof}

\begin{lemd}\label{subsub2}
Let $\rho:M_1\rightarrow M_2$ be a surjective
submersion of smooth manifolds. Let $Z_1$ be a submanifold of
$M_1$ which is a union of fibres of $\rho$. Then $Z_2:=\rho(Z_1)$
is submanifold of $M_2$, and the restriction $\rho_0:
Z_1\rightarrow Z_2$ of $\rho$ is also a surjective submersion.
Furthermore, if $Z_1$ is closed in $M_1$, then $Z_2$ is closed in
$M_2$.
\end{lemd}
\begin{proof}
Write
\[
    n_1:=\dim M_1\geq n_2:=\dim M_2.
\]
Take two open embeddings $i_1: \R^{n_1}\hookrightarrow M_1$ and
$i_2: \R^{n_2}\hookrightarrow M_2$ such that the diagram
\[
  \begin{CD}
           \R^{n_1} @> i_1>> M_1\\
            @V \rho' VV           @VV \rho V \\
           \R^{n_2}@>i_2>> M_2\\
  \end{CD}
\]
commutes, where $\rho'$ is the projection to the first $n_2$
coordinates.

Write
\[
   Z_1':=i_1^{-1}(Z_1)\subset \R^{n_1}\quad \textrm{and}\quad
   Z_2':=i_2^{-1}(Z_2)\subset \R^{n_2}.
\]
Then $Z_1'$ is a submanifold of $\R^{n_1}$ which is a union of
fibres of $\rho'$. The condition that $Z_1$ is a union of fibres
of $\rho$ implies that
\[
  Z_2'=\rho'(Z_1').
\]

By the local triviality of submersions, it suffices to prove the
lemma for $\rho'$ and $Z_1'$. The latter is now immediate in view
of Lemma \ref{subsub1} and the fact that
\[
 Z_1'=Z_2'\times \R^{n_1-n_2}.
\]

\end{proof}

By setting
\[
  (\rho:M_1\rightarrow M_2)=(\rho_{\fM}: H\times \fM\rightarrow M),
\]
and
\[
  Z_1=H\times \fZ
\]
in Lemma \ref{subsub2}, we have the following

\begin{lemd}\label{localization1}
Let $\fM$ be an $H$ slice of $M$, and let $\fZ$ be a relatively $H$
stable submanifold of $\fM$. Then $Z=H \fZ$ is a submanifold of $M$,
and $\fZ$ is an $H$ slice of $Z$. Furthermore if $\fZ$ is closed in
$\fM$, then $Z$ is closed in $M$.
\end{lemd}

Lemma \ref{localization1} will be used extensively in Sections \ref{sm4} and \ref{gl6}.

\vsp

Now assume that $\fM$ is a local $H$ slice of $M$, and $H_{\fM}$
is a closed subgroup of $H$ which leaves $\fM$ stable. Let $H$ act
on $H\times \fM$ by left multiplication on the first factor, and
let $H_{\fM}$ act on $H\times \fM$ by
\[
   g(h,x)=(ghg^{-1}, gx), \quad g\in H_{\fM}, \, h\in H, x\in
   \fM.
\]
Then the submersion $\rho_{\fM}$ is $H$ intertwining as well as
$H_{\fM}$ intertwining. Therefore the pulling back yields a linear
map
\[
  \rho_{\fM}^*: \con^{-\infty}_\chi(M)\rightarrow \con^{-\infty}_\chi(H\times
 \fM)\cap \con^{-\infty}_{\chi_{\fM}}(H\times\fM),
\]
where $\chi_\fM=\chi|_{H_\fM}$. By the Schwartz Kernel Theorem and
the fact that every invariant distribution on a Lie group is a
scalar multiple of the Haar measure (\cite[8.A]{W1}), we have
\[
  \con^{-\infty}_\chi(H\times
 \fM)=\chi\otimes \con^{-\infty}(\fM).
\]
Consequently,
\[
  \con^{-\infty}_\chi(H\times
 \fM)\cap \con^{-\infty}_{\chi_{\fM}}(H\times\fM)=\chi\otimes
 \con^{-\infty}_{\chi_{\fM}}(\fM).
\]

We shall record this as

\begin{lemd}\label{restriction}
There is a well-defined map which is called the restriction to
$\fM$:
\[
  \con^{-\infty}_\chi(M)\rightarrow
  \con^{-\infty}_{\chi_{\fM}}(\fM),\quad f\mapsto f|_\fM
\]
by requiring that
\[
  \rho_{\fM}^*(f)=\chi\otimes f|_\fM.
\]
The map is injective when $\fM$ is an $H$ slice.
\end{lemd}

\subsection{Nash manifolds and tempered generalized functions}
\label{snash}

We begin with a review of basic concepts and properties of Nash
manifolds, in which the notion of tempered generalized functions is
defined. Our main reference on Nash manifolds is \cite{Sh}, and
temperedness is discussed in \cite{du,AG-Nash}.

\vsp \noindent {\bf Remark}: We will use Fourier transforms
implicitly in Section \ref{gl222}, and explicitly in Section
\ref{sgl3}. Fourier transforms are only defined for tempered
generalized functions. This is the main reason that we work with
tempered generalized functions instead of arbitrary generalized
functions.

\vsp

Recall that the collection $\mathcal S\mathcal A_n$ of semialgebraic
subsets of $\R^n$ is the smallest set with the following properties:
\begin{enumerate}
\item[(a)]
 every element of $\mathcal S\mathcal A_n$ is a subset of $\R^n$;
\item[(b)]
 for every real polynomial function $p$ on $\R^n$, we have
\[
  \{x\in \R^n\mid p(x)>0\}\in \mathcal S\mathcal A_n;
\]
 \item[(c)]
  $\mathcal S\mathcal A_n$ is closed under the operation of taking intersection, and taking
  complement in $\R^n$.
\end{enumerate}

A Nash manifold of dimension $n$ is a manifold $M$, together with a
collection $\mathcal{N}$, whose members are called Nash charts, such
that the followings hold:
\begin{enumerate}
\item[(a)]
 every Nash chart has the form $(\phi,U,U')$, where $U$
 is an open semialgebraic subset of $\R^n$, $U'$ is an open subset
 of $M$, and $\phi:U\rightarrow U'$ is a diffeomorphism;
\item[(b)]
 every two Nash charts $(\phi_1,U_1,U'_1)$ and $(\phi_2,U_2,U'_2)$ are Nash compatible, i.e., the graph of the
 diffeomorphism
 \[
  \phi_2^{-1}\circ \phi_1: \phi_1^{-1}(U_1'\cap U_2')\rightarrow \phi_2^{-1}(U_1'\cap U_2')
 \]
 is semialgebraic;
\item[(c)]
  for every triple $(\phi,U,U')$ as in (a), if it is Nash compatible
  with all Nash charts, then itself is a Nash chart;
\item[(d)]
  there are finitely many Nash charts $(\phi_i,U_i,U'_i)$, $i=1,2,\cdots r$, such that
\[
   M=U'_1\cup U'_2\cup \cdots \cup U_r'.
\]
\end{enumerate}
 A subset $Z$ of $M$ is
called semialgebraic if
\[
  \phi^{-1}(Z\cap U')\textrm{ is semialgebraic in $\R^n$ }
\]
for all Nash chart $(\phi,U,U')$ of $M$. A Nash manifold is either
the empty set or a nonempty Nash manifold of dimension $n\geq 0$.
A submanifold of a Nash manifold which is semialgebraic is called
a Nash submanifold, which is automatically a Nash manifold. The
product of two Nash manifolds is again a Nash manifold. A smooth
map $\phi: M_1\rightarrow M_2$ of Nash manifolds is called a Nash
map if its graph is semialgebraic in $M_1\times M_2$. (A Nash map
always sends a semialgebraic set to a semialgebraic set.) A Nash
function on a Nash manifold $M$ is a Nash map from $M$ to $\C$,
and a differential operator $D$ on $M$ is called Nash if $D(f)$ is
Nash for every Nash function $f$ on every Nash open submanifold of
$M$.

A Nash group is a group as well as a Nash manifold so that the group
operations are Nash maps. A Nash action of a Nash group on a Nash
Manifold is defined similarly.

\vsp We proceed to our discussion on the notion of tempered
generalized functions on a Nash manifold. A smooth function $f$ on a
semialgebraic open subset $U$ of $\R^n$ is called a Schwartz
function if $D(f)$ is bounded for every Nash differential operator
$D$ on $U$. Denote by ${\mathcal S} (U)$ the Fr\'{e}chet space of
Schwartz functions on $U$. Now let $M$ be a Nash manifold of
dimension $n$. Pick a covering of $M$ by Nash charts
$(\phi_i,U_i,U'_i)$, $i=1,2,\cdots ,r$. By extending to zero outside
$U'_i$, $\phi_i$ induces a continuous linear map
\[
  (\phi_i)_*: {\mathcal S} (U_i)\rightarrow \con^{\infty}(M).
\]
The Fr\'{e}chet space of Schwartz functions on $M$, denoted by
${\mathcal S}(M)$, is then defined to be the image of the map
\[
 \oplus(\phi_i)_*:\oplus_{i=1}^r {\mathcal S}(U_i)\rightarrow \con^{\infty}(M),
\]
equipped with the quotient topology of $\oplus_{i=1}^r {\mathcal
S}(U_i)$. This definition is independent of the covering we
choose. One may similarly define the Fr\'{e}chet space of Schwartz
densities. Denote by $\con^{-\xi}(M)$ its strong dual, whose
members are called tempered generalized functions. All tempered
generalized functions are generalized functions.

\vsp Now let $H$ be a Nash group, with a Nash action on a Nash
manifold $M$. For any character $\chi$ on $H$, we set
\begin{equation}
  \con^{-\xi}_{\chi}(M)=\con^{-\infty}_{\chi}(M)\cap \con^{-\xi}(M).
\end{equation}

Let $N$ be a Nash manifold, and let $\phi:M\rightarrow N$ be an
$H$ invariant Nash map. We record the following obvious fact as a
lemma.

\begin{lemd}\label{nash}
If $\con^{-\xi}_{\chi}(M)=0$, then
$\con^{-\xi}_{\chi}(\phi^{-1}(N'))=0$ for  all Nash open
submanifold $N'$ of $N$.
\end{lemd}

Let $\fM$, $H_\fM$ and $\chi_\fM$ be as in Lemma \ref{restriction}.
If furthermore $\fM$ is a Nash submanifold of $M$, and $H_\fM$ is a
Nash subgroup of $H$, then the restriction map sends
$\con^{-\xi}_{\chi}(M)$ into $\con^{-\xi}_{\chi_\fM}(\fM)$.

\section{Metrical properness and unipotent $\chi$-incompatibility}
\label{UXM}

\subsection{Metrical properness}

This notion requires that the manifold $M$ is pseudo Riemannian,
i.e., the tangent spaces are equipped with a smoothly varying
family $\{\la\,,\,\ra_x:x\in M\}$ of nondegenerate symmetric
bilinear forms.

\begin{dfn}
\label{dmp} \begin{itemize} \item[(a)] A submanifold $Z$ of a
pseudo Riemannian manifold $M$ is said to be metrically proper if
for all $z\in Z$, the tangent space $\oT_z(Z)$ is contained in a
proper nondegenerate subspace of $\oT_z(M)$. \item[(b)] A
differential operator $D\in \Diff(M)_2$ is said to be of Laplacian
type if for all $x\in M$, the principal symbol
\[
  \sigma_2(D)(x)=u_1 v_1+u_2 v_2+\cdots+u_m v_m,
\]
where $u_1,u_2,\cdots,u_m$ is a basis of the tangent space
$\oT_x(M)$, and $v_1,v_2,\cdots,v_m$ is the dual basis in $\oT_x(M)$
with respect to $\la\,,\,\ra_x$.
\end{itemize}
\end{dfn}

Note that a Laplacian type differential operator is transversal to
any metrically proper submanifold, from its very definition.
Therefore the following is a special case of Lemma
\ref{tvanishing2}.

\begin{lemd}\label{tvanishing3}
Let $Z$ be a metrically proper submanifold of $M$, and let $D$ be a
Laplacian type differential operator on $M$. Then
\[
   \con^{-\infty}(M;Z;D)=0.
\]
\end{lemd}

\subsection{Unipotent $\chi$-incompatibility}
As in Section \ref{inv}, let $H$ be a Lie group with a character
$\chi$ on it, acting smoothly on a manifold $M$. If a locally
closed subset $Z$ of $M$ is $H$ stable, denote by
$\con^{-\infty}_\chi(M;Z)$ the space of all $f$ in
$\con^{-\infty}(M;Z)$ which are $\chi $-equivariant. We shall use
similar notations (such as $\con^{-\infty}_\chi(M;D)$ and
$\con^{-\infty}_\chi(M;Z;D)$) without further explanation.

\begin{dfnd}
\label{duc} An $H$ stable submanifold $Z$ of $M$ is said to be
unipotently $\chi$-incompatible if for every $z_0\in Z$, there is a
local $H$ slice $\fZ$ of $Z$, containing $z_0$, and a smooth map
$\phi: \fZ\rightarrow H$ such that the followings hold for all $z\in
\fZ$:
\begin{itemize}
 \item[(a)]
   $\phi(z)z=z$, and
 \item[(b)]
   the linear map
   \[
     \oT_z(M)/\oT_z(Z)\rightarrow \oT_z(M)/\oT_z(Z)
   \]
      induced by the action of $\phi(z)$ on $M$ is unipotent;
      \item[(c)]
$\chi(\phi(z))\neq 1$.

\end{itemize}
\end{dfnd}

The following lemma will be important for our later considerations.

\begin{lemd}\label{localization3}
Let $Z$ be an $H$ stable submanifold of $M$ which is unipotently
$\chi$-incompatible. Then $\con_\chi^{-\infty}(M;Z)=0$.
\end{lemd}

By using a well-known result of L. Schwartz on the filtration of the
sheaf of generalized functions with supports in a submanifold, Lemma
\ref{localization3} is implied by the following

\begin{sublemd}\label{localization4}
Let $\fZ$ be an $H$ slice of an $H$ manifold $Z$. Let $E$ be an $H$
equivariant smooth complex vector bundle over $Z$, of finite rank.
Assume that there is a smooth map $\phi: \fZ\rightarrow H$ such that
for all $z\in \fZ$,
\begin{itemize}
 \item[(a)]
    $\phi(z)z=z$, and
  \item[(b)]
   the linear map
   \[
     \phi(z): E_z\rightarrow E_z
   \]
   is unipotent, where $E_z$ is the fibre of $E$ at $z$;
   \item[(c)]
   $\chi(\phi(z))\neq 1$.

\end{itemize}
Then
\[
  \Gamma^{-\infty}_\chi(E)=0.
\]
\end{sublemd}

Here and as usual, ``$\Gamma^{-\infty}$" stands for the space of
generalized sections. (We omit its definition since it is a
straightforward generalization of the notion of generalized
functions, in Section \ref{sub2.1}.) The space
$\Gamma^{-\infty}_\chi(E)$ consists of all $f\in
\Gamma^{-\infty}(E)$ such that
\begin{equation}\label{fhx}
   f(hx)=\chi(h)h(f(x)),\quad\textrm{ for all }    h\in H.
\end{equation}
The meaning of (\ref{fhx}) will be made clear in the following
proof.
\begin{proof}
As in the case of generalized functions, define the pulling back
\[
    \rho_{\fZ}^*: \Gamma^{-\infty}(E)\rightarrow \Gamma^{-\infty}(\tilde E),
\]
of the action map
\[
    \rho_{\fZ}: H\times \fZ\rightarrow Z,
\]
which continuously extends the usual pulling back of smooth
sections. Here $\tilde E$ is the pulling back of $E$ via $\rho_\fZ$,
which is obviously an $H$ equivariant vector bundle over $H\times
\fZ$. Note that the bundle $\tilde{E}|_{\{e\}\times \fZ}$ is
identified with $E|_\fZ$. The restriction $f|_\fZ\in
\Gamma^{-\infty}(E|_\fZ)$ of an element $f\in
\Gamma^{-\infty}_\chi(E)$ is then specified by
\begin{equation}\label{equalitygs}
  \tilde{f}(h,z)=\chi(h)h f|_\fZ(z), \quad\textrm{where } \tilde
  f=\rho_\fZ^*(f).
\end{equation}
C.f. Lemma \ref{restriction}.
Here we caution the reader due to the fact that we are dealing with
generalized (as opposed to smooth) sections. The formula
(\ref{equalitygs}) is to be understood as an equality in
$\Gamma^{-\infty}(\tilde E)$. The righthand side makes sense since
the map
\[
   \Gamma^{\infty}(E|_\fZ)\rightarrow \Gamma^{\infty}(\tilde
   E),\quad f'(z)\mapsto \chi(h)h f'(z)
\]
of smooth sections extends continuously to a (well-defined) map
\begin{equation}\label{gamres}
   \Gamma^{-\infty}(E|_\fZ)\rightarrow \Gamma^{-\infty}(\tilde
   E),\quad f'(z)\mapsto \chi(h)h f'(z)
\end{equation}
of generalized sections. Similarly, all the equalities below, which
are obvious when $f|_\fZ$ is a smooth section, make sense and hold
true by a continuity argument.

Condition (a) implies
\[
 \tilde{f}(h\phi(z),z)=\tilde{f}(h,z),
\]
and (\ref{equalitygs}) implies
\[
  \tilde{f}(h\phi(z),z)=\chi(h\phi(z))h\phi(z)f|_\fZ(z).
\]
Therefore
\begin{equation}\label{echih}
  \chi(h\phi(z))h\phi(z)f|_\fZ(z)=\chi(h)hf|_\fZ(z).
\end{equation}
Since the map (\ref{gamres}) is injective, (\ref{echih}) implies
that
\[
 (\chi(\phi(z))\phi(z)-1_{E_z})f|_\fZ(z)=0,
\]
where $\phi(z)$ is viewed as a linear automorphism of $E_z$, and
$1_{E_z}$ is the identity map of $E_z$. Conditions (b) and (c) imply
that $\chi(\phi(z))\phi(z)-1_{E_z}$ is invertible on $E_z$ and so
\[
  f|_\fZ(z)=(\chi(\phi(z))\phi(z)-1_{E_z})^{-1}(\chi(\phi(z))\phi(z)-1_{E_z})f|_\fZ(z)=0,
\]
which implies that $f=0$.
\end{proof}

Recall the notion of a Nash group from Section \ref{snash}. It is
said to be unipotent if it is Nash isomorphic to a connected closed
subgroup of some $U_n$, where $U_n$ is the Nash group of unipotent
upper triangular real matrices of size $n$. An element of a Nash
group is said to be (Nash) unipotent if it is contained in a
unipotent Nash closed subgroup. We note that the general linear
group $\GL_n(\K)$ is Nash and an element of $\GL_n(\K)$ is (Nash)
unipotent if and only if it is unipotent in the usual sense, i.e.,
is a unipotent linear transformation.

If $H$, $M$ and the action of $H$ on $M$ are all Nash, then an $H$
stable submanifold $Z$ of $M$ is unipotently $\chi$-incompatible
if the following holds: for every point $z_0\in Z$, there is a
local $H$ slice $\fZ$ of $Z$, containing $z_0$, and a smooth map
$\phi: \fZ\rightarrow H$ such that, for all $z\in \fZ$,
\begin{enumerate}
    \item[(a)]
  $\phi(z)z=z$, and
   \item[(b)]
   $\phi(z)$ is (Nash) unipotent;
   \item[(c)]
   $\chi(\phi(z))\neq 1$.
\end{enumerate}
The reason for this is that the hypothesis of Nash action ensures
that the map
\[
\oT_{z}(M)\rightarrow \oT_{z}(M),
\]
induced by the action of the unipotent element $\phi(z)$, is
unipotent. This implies condition (b) in Definition \ref{duc}.

\subsection{$\oU_{\chi}\oM$ property}
As before, let $H$ be a Lie group acting smoothly on a manifold
$M$, and let $\chi$ be a character on $H$. We further assume that
$M$ is a pseudo Riemannian manifold.

\begin{dfnd}
\label{dum} We say that an $H$ stable locally closed subset $Z$ of
$M$ has $\oU_{\chi}\oM$ property if there is a finite filtration
\[
  Z=Z_0\supset Z_1\supset\cdots\supset Z_k\supset Z_{k+1}=\emptyset
\]
of $Z$ by $H$ stable closed subsets of $Z$ such that each
$Z_i\setminus Z_{i+1}$ is a submanifold of $M$ which is either
unipotently $\chi$-incompatible or metrically proper in $M$.
\end{dfnd}

As a combination of Lemma \ref{tvanishing3} and Lemma
\ref{localization3}, we have
\begin{lemd}
\label{UchiM} Let $D$ be a  differential operator on $M$ of
Laplacian type. Let $Z$ be an $H$ stable closed subset of $M$
having $\oU_{\chi}\oM$ property. Then
\[
  \con^{-\infty}_\chi(M;Z;D)=0.
\]
\end{lemd}

\section{Small submanifolds of $\GL_6$}
\label{smallsub}

We return to the group $G=\GL_6(\K)$. Recall from the Introduction
the subgroup $S$ and its character $\chi_S$. From now on, we set
\begin{equation}
\label{Handchi}
  H=S\times S,\quad\textrm{and}\quad \chi=\chi _S \otimes \chi
  _S.
\end{equation}
Let $H$ act on $G$ by
\[
  (g_1,g_2)x=g_1 x g_2^\tau.
\]
Our main object of concern is the space $\con_\chi^{-\infty}(G)$.

For $x\in G$, define its rank matrix
\[
   \operatorname{R}(x)=\left[\begin{array}{cc}
  \rank_{4\times 4}(x)&\rank_{4\times 2}(x)\\
  \rank_{2\times 4}(x)&\rank_{2\times 2}(x)\\
  \end{array}\right],
\]
where $\rank_{i\times j}(x)$ is the rank of the lower right
$i\times j$ block of $x$. Then $R(x)$ takes the following $21$
possible values \cite{Nien}:
\[
   \left[\begin{array}{cc}
    4&2\\
    2&2\\
  \end{array}\right],
   \left[\begin{array}{cc}
    4&2\\
    2&1\\
  \end{array}\right],
   \left[\begin{array}{cc}
    4&2\\
    2&0\\
  \end{array}\right],
  \left[\begin{array}{cc}
    3&2\\
    2&2\\
  \end{array}\right],
  \left[\begin{array}{cc}
    3&2\\
    2&1\\
  \end{array}\right],
\]
\[
   \left[\begin{array}{cc}
    3&2\\
    1&1\\
  \end{array}\right],
   \left[\begin{array}{cc}
    3&1\\
    2&1\\
  \end{array}\right],
   \left[\begin{array}{cc}
    3&1\\
    1&1\\
  \end{array}\right],
  \left[\begin{array}{cc}
    3&2\\
    1&0\\
  \end{array}\right],
  \left[\begin{array}{cc}
    3&1\\
    2&0\\
  \end{array}\right],
  \left[\begin{array}{cc}
    3&1\\
    1&0\\
  \end{array}\right],
\]
\[
  \left[\begin{array}{cc}
    2&2\\
    2&2\\
  \end{array}\right],
  \left[\begin{array}{cc}
    2&2\\
    1&1\\
  \end{array}\right],
      \left[\begin{array}{cc}
    2&1\\
    2&1\\
  \end{array}\right],
   \left[\begin{array}{cc}
    2&1\\
    1&1\\
  \end{array}\right],
\]
\[
  \left[\begin{array}{cc}
    2&2\\
    0&0\\
  \end{array}\right],
  \left[\begin{array}{cc}
    2&0\\
    2&0\\
  \end{array}\right],
  \left[\begin{array}{cc}
    2&1\\
    1&0\\
  \end{array}\right],
   \left[\begin{array}{cc}
    2&1\\
    0&0\\
  \end{array}\right],
  \left[\begin{array}{cc}
    2&0\\
    1&0\\
  \end{array}\right],
  \left[\begin{array}{cc}
    2&0\\
    0&0\\
  \end{array}\right].
\]
For $R$ one of the above, denote
\begin{equation}
\label{GR}
  G_{R}=\{x\in G\mid \operatorname{R}(x)=R\}.
\end{equation}
Then
\[ G=\bigsqcup_{R}G_R
\]
is the decomposition of $G$ into $P$-$P^\tau$ double cosets.
Define an open submanifold
\begin{equation*}
\begin{aligned}
  &G'=\{x\in G\mid \rank_{2\times 4}(x)=\rank_{4\times 2}(x)=2,\\
  &\phantom{G'=\{x\in G\mid } \ \rank_{2\times 2}(x)\geq 1,\, \rank_{4\times 4}(x)\geq
  3\}.
\end{aligned}
\end{equation*}
Then we have \begin{equation} \label{dg'} G'=\bigsqcup G_R,
\end{equation}
where $R$ in the union runs through the following four matrices
\begin{equation}\label{fourm}
 \left[\begin{array}{cc}
    4&2\\
    2&2\\
  \end{array}\right],
   \left[\begin{array}{cc}
    4&2\\
    2&1\\
  \end{array}\right],
    \left[\begin{array}{cc}
    3&2\\
    2&2\\
  \end{array}\right],
  \left[\begin{array}{cc}
    3&2\\
    2&1\\
     \end{array}\right].
\end{equation}

Let $\Delta $ be the Casimir operator on $G$, as in the
Introduction. The goal of this section is to prove the following

\begin{prp}\label{mp5}
Let $f\in \con^{-\infty}_{\chi}(G)$. If $f$ is an eigenvector of
$\Delta$, and $f$ vanishes on $G'$, then $f=0$.
\end{prp}

Set
\begin{equation}
\label{xleft}
  x_{\oleft}=\left[
               \begin{array}{ccc}
                  0&I_2&0\\
                  0&0&I_2\\
                  0&0&0\\
                \end{array}
            \right]\in \gl_6(\K)\quad \textrm{ and } x_{\oright}=x_{\oleft}^\tau.
\end{equation}
Denote by $X_{\oleft}$ the left invariant vector field on $G$
whose tangent vector at $x$ is $xx_{\oleft}$, and by $X_{\oright}$
the right invariant vector field on $G$ whose tangent vector at
$x$ is $x_{\oright}x$.

The key to Proposition \ref{mp5} is the following transversality
result. We shall divide it into a number of lemmas (Lemmas
\ref{smallm0}, \ref{smallm2}, \ref{smallm3}, \ref{smallm4}).

\begin{prpp}\label{smallm} Assume that $R$ is not one of the four matrices in
(\ref{fourm}).
Then either $X_{\oleft}$ or $X_{\oright}$ is transversal to the
double coset $G_R$.
\end{prpp}

\begin{lemp}\label{smallm0}
If the lower right entry of $R$ is zero, then $X_{\oleft}$ is
transversal to $G_R$.
\end{lemp}
\begin{proof}
Assume that there is an
\[
 x=\left[\begin{array}{ccc}
    x_{11}&x_{12}&x_{13}\\
     x_{21}&x_{22}&x_{23}\\
      x_{31}&x_{32}&0\\
  \end{array}\right]\in G_R
\]
such that
\[
  X_{\oleft}(x)\in \oT_x(G_R),
\]
i.e.,
\[
  \left[\begin{array}{ccc}
   0&x_{11}&x_{12}\\
    0& x_{21}&x_{22}\\
     0& x_{31}&x_{32}\\
  \end{array}\right]\in \Lie(P)x+x\Lie(P^\tau).
\]
Note that the lower right $2\times 2$ block of very element of
$\Lie(P)x+x\Lie(P^\tau)$ is $0$. Therefore $x_{32}=0$, which
further implies that  the lower right $2\times 4$ block of very
element of $\Lie(P)x+x\Lie(P^\tau)$ is $0$. Therefore $x_{31}=0$.
This contradicts the fact that $x$ is invertible.
\end{proof}

The following lemma provides a technical simplification.

\begin{lemp}\label{smallm1}
Let $x,y$ be two matrices in $G_R$ such that $PxS^\tau=PyS^\tau$.
Then
\[
  X_{\oleft}(x)\in \oT_x(G_R)\quad\textrm{if and only if}\quad X_{\oleft}(y)\in
  \oT_y(G_R).
\]
\end{lemp}
\begin{proof}
Write
\[
  y=pxq, \quad p\in P, q\in S^\tau,
\]
and assume that $X_{\oleft}(x)\in \oT_x(G_R)$, i.e.,
\[
  xx_{\oleft}\in \Lie(P)x+x\Lie(P^\tau).
\]
One easily checks that
\[
  x_{\oleft}q-qx_{\oleft}\in \Lie(P^\tau).
\]
Therefore
\[\begin{aligned}
   yx_{\oleft}&=px q x_{\oleft}\\
   &\in pxx_{\oleft}q+ px\Lie(P^\tau)\\
   &\subset p(\Lie(P)x+x\Lie(P^\tau))q+px\Lie(P^\tau)\\
   &=\Lie(P)y+y\Lie(P^\tau).
   \end{aligned}
   \]
The last equality holds because
\[
  p\Lie(P)=\Lie(P)p=\Lie(P), \textrm{}\quad
  q\Lie(P^\tau)=\Lie(P^\tau)q=\Lie(P^\tau).
\]
\end{proof}

\begin{lemp}\label{smallm2}
If the second row of $R$ is  $[1\,\,\,\, 1]$, then $X_{\oleft}$ is
transversal to $G_R$.
\end{lemp}
\begin{proof}
Let $R$ be as in the lemma. Then every matrix in $G_R$ is in the
same $P$-$S^\tau$ double coset with a matrix of the form
\[
  x=\left[\begin{array}{ccc}
    x_{11}&x_{12}&x_{13}\\
     x_{21}&x_{22}&x_{23}\\
      x_{31}&0&\delta_2\\
  \end{array}\right]\in G_R,
\]
where
\[
  \delta_2=\left[\begin{array}{cc}
    0&0\\
    0&1\\
  \end{array}\right].
\]

Assume that
\[
   X_{\oleft}(x)\in \oT_x(G_R),
\]
i.e.,
\[
  \left[\begin{array}{ccc}
    0&x_{11}&x_{12}\\
    0& x_{21}&x_{22}\\
     0& x_{31}&0\\
  \end{array}\right]\in \Lie(P)x+x\Lie(P^\tau).
\]
Note that the middle $2\times 2$ block of the last two rows of
very matrix in $\Lie(P)x+x\Lie(P^\tau)$ has the form
\[
  \delta_2 u, \quad u\in \gl_2(\K).
\]
This implies that the first row of $x_{31}$ is zero, and
consequently, the fifth row of $x$ is zero, which contradicts the
fact that $x$ is invertible.
\end{proof}

Similarly, we have
\begin{lemp}\label{smallm3}
If the second column of $R$ is $\left[\begin{array}{c}
    1\\
    1\\
  \end{array}\right]$, then $X_{\oright}$ is transversal
to $G_R$.
\end{lemp}

\begin{lemp}\label{smallm4}
If
\[
  R=\left[\begin{array}{cc}
    2&2\\
    2&2\\
  \end{array}\right],
\]
then $X_{\oleft}$ is transversal to $G_R$.
\end{lemp}
\begin{proof}
Every matrix in $G_R$ is in the same $P$-$S^\tau$ double coset
with a matrix of the form
\[
  x=\left[\begin{array}{ccc}
    x_{11}&x_{12}&0\\
     x_{21}&0&0\\
      0&0&1\\
  \end{array}\right]\in G_R.
\]
Assume that
\[
   X_{\oleft}(x)\in \oT_x(G_R),
\]
i.e.,
\[
  \left[\begin{array}{ccc}
    0&x_{11}&x_{12}\\
    0& x_{21}&0\\
     0&0&0\\
  \end{array}\right]\in \Lie(P)x+x\Lie(P^\tau).
\]
Note that the central $2\times 2$ block of every matrix in
$\Lie(P)x+x\Lie(P^\tau)$ is zero, which implies that $x_{21}=0$.
This contradicts the fact that $x$ is invertible.
\end{proof}

The proof of Proposition \ref{smallm} is now finished. \vsp

\begin{lemp}\label{mp4} There exists a nonzero number $c$, an
element $\lambda\in\K^\times$, and a differential operator
$D_{\oleft}$ on $G$, which is tangential to every $P$-$P^\tau$
double coset of $G$, such that
\[
  \Delta f=(c X_{\oleft}(\lambda)+D_{\oleft})f
  \]
for all $f\in \con^{-\infty}_{\chi}(G)$. Here
$X_{\oleft}(\lambda)$ is the left invariant vector field on $G$
whose tangent vector at $x\in G$ is $\lambda x x_{\oleft}$, and
$x_{\oleft}$ is given in (\ref{xleft}). The same is true if one
replaces ``left" by ``right" everywhere.
\end{lemp}

\begin{proof}
The Lie algebra $\g$ of $G$ has a decomposition
\[
  \g=\n+\l+\n^\tau,
\]
where $\n$ is the Lie algebra of the unipotent radical $N$ of $P$,
and $\l$ is the Lie algebra of the Levi factor $\GL_2(\K)\times
\GL_2(\K)\times \GL_2(\K)$. Recall that $\g$ is equipped with the
real trace form. Let $X_1, X_2,\cdots, X_r$ be a basis of $\n$,
and write
\[
  \Delta_1=X_1 X_1'+X_2X_2'+\cdots +X_r X_r'\in \oU(\g),
\]
where $X_1', X_2',\cdots, X_r'$ is the dual basis of $X_1,
X_2,\cdots, X_r$ in $\n^\tau$. Note that $\Delta_1$ is independent
of the choice of basis of $\n$. We identify elements of $\oU(\g)$
with left invariant (real) differential operators on $G$ as usual.
It is then easy to see that
\begin{equation}\label{delta1}
  \Delta-2 \Delta_1 \in \oU(\l).
\end{equation}

Let $d\chi_S$ be the differential of the character $\chi_S$. Write
\[
  \chi_{\n^\tau}(X)=d\chi_S(-X^\tau), \quad X\in \n^\tau,
\]
which defines a character of $\n^\tau$. Then every generalized
function $f\in \con^{-\infty}_{\chi}(G)$ satisfies
\[
  Xf=-\chi_{\n^\tau}(X)f,\quad \textrm{ for all }X\in \n^\tau.
\]

Now choose $X_1$ to be perpendicular to the kernel of
$\chi_{\n^\tau}$. This is unique up to a multiple in $\R^\times$,
and has the form $X_{\oleft}(\lambda)$ for some $\lambda\in
\K^\times$. This choice of $X_1$ also implies that
\[
  \chi_{\n^\tau}(X_2')=\chi_{\n^\tau}(X_3')=\cdots=\chi_{\n^\tau}(X_r')=0,
\]
and $\chi_{\n^\tau}(X_1')$ is a nonzero number.

Therefore
\begin{equation}\label{delta2}
  \Delta_1 f=-\chi_{\n^\tau}(X_1') X_{\oleft}(\lambda) f\quad\textrm{for all }f\in
  \con^{-\infty}_{\chi}(G).
\end{equation}
Equations (\ref{delta1}) and (\ref{delta2}) will now imply the
lemma, in view of the fact that a differential operator in
$\oU(\l)$ is tangential to every $P$-$P^\tau$ double coset.
\end{proof}

Now we are ready to prove Proposition \ref{mp5}. Take a sequence
\[
  G'=G_{\open}^4\subset
  G_{\open}^5\subset\cdots
  \subset G_{\open}^{20}\subset
  G_{\open}^{21}=G,
\]
of open subsets of $G$ so that every difference
$G_{\open}^{i}\setminus G_{\open}^{i-1}$ is a $P$-$P^\tau$ double
coset, $i=5,6,\cdots,21$. Denote by $f_i$ the restriction of $f\in
\con^{-\infty}_{\chi}(G)$ to $G_{\open}^{i}$. We shall use
induction to show that all $f_i$'s are zero. Thus assume that
$f_{i-1}=0$.

By Proposition \ref{smallm}, either $X_{\oleft}$ or $X_{\oright}$
is transversal to $G_{\open}^{i}\setminus G_{\open}^{i-1}$.
Without loss of generality assume that $X_{\oleft}$ is transversal
to $G_{\open}^{i}\setminus G_{\open}^{i-1}$. Lemma \ref{mp4}
implies that
\[
  (X_{\oleft}(\lambda)+D)f_i=0,
\]
where $D$ is a differential operator on $G_{\open}^{i}$ which is
tangential to $G_{\open}^{i}\setminus G_{\open}^{i-1}$. It is
clear that $X_{\oleft}$ is transversal to $G_{\open}^{i}\setminus
G_{\open}^{i-1}$ will imply the same for $X_{\oleft}(\lambda)$.
Invoking Lemma \ref{tvanishing2}, we see that $f_i=0$.

\section{A submanifold $Z_4$ of $\GL_4\times \GL_2$}\label{sm4}

As always, we equip $G=\GL_6(\K)$ with the bi-invariant pseudo
Riemannian metric whose restriction to $\oT_e(G)=\gl_6(\K)$ is the
real trace form $\la\,,\,\ra_{\R}$, given in (\ref{rtrace}).

As in the Introduction, write $G_{4,2}=\GL_4(\K)\times \GL_2(\K)$,
which embeds into $G$ in the usual way. Then $G_{4,2}$ is a
nondegenerate submanifold of $G$, with $\oT_e(M)=\gl_4(\K)\times
\gl_2(\K)$. Thus $G_{4,2}$ is itself a pseudo Riemannian manifold.

Denote
\[
   S_{4,2}=(\GL_4(\K)\times \GL_2(\K))\cap S =\left\{\,\left[\begin{array}{ccc} a&b&0\\ 0&a&0\\0&0&a\\
                   \end{array}\right]\in G\right\},
\]
\begin{equation}
  H_{4,2}=S_{4,2}\times S_{4,2} \subset H=S\times S,
\end{equation}
and the character $\chi_{4,2}=\chi |_{H_{4,2}}$.

Let $Z_4$ be the following $H_{4,2}$ stable submanifold of
$G_{4,2}$:
\begin{equation}
\label{z4}
 Z_4=\left\{\,\left[\begin{array}{ccc} a_{11}&a_{12}&0\\
    a_{21}&a_{22}&0\\
     0&0&y
                   \end{array}\right]\in G_{4,2}\mid
   \textrm{$y^{-1}a_{22}$ is nilpotent and nonzero}\right\}.
\end{equation}

The purpose of this section is to prove the following proposition.
This will take a number of steps.

\begin{m4}\label{m4}
As an $H_{4,2}$ submanifold of $G_{4,2}$, $Z_4$ has
$\oU_{\chi_{4,2}}\oM$ property.
\end{m4}

Denote by $\fZ_4$ all matrices in $G_{4,2}$ of the form
\begin{equation}\label{xequals}
  x=\left[\begin{array}{cccccc} x_{11}&x_{12}&x_{13}&0&0&0\\
  x_{21}&x_{22}&x_{23}&0&0&0\\
  x_{31}&x_{32}&0&0&0&0\\
  0&0&0&1&0&0\\
   0&0&0&0&0&1\\
    0&0&0&0&1&0\\
  \end{array}\right].
  \end{equation}

\begin{m4slice}
The submanifold $\fZ_4$ is an $H_{4,2}$ slice of $Z_4$.
\end{m4slice}

\begin{proof}
Let $x\in Z_4$ be as in (\ref{z4}). Define
\[
  \bar{\phi}(x)=\left[\begin{array}{cc}
    a_{22}&0\\
     0&y
                   \end{array}\right]\in \gl_4(\K).
\]
Note that $\bar{\phi}(\fZ_4)$ consists of a single matrix
\[
  \bar{x}_0=\left[\begin{array}{cccc}
  0&0&0&0\\
  0&1&0&0\\
  0&0&0&1\\
  0&0&1&0\\
  \end{array}\right].
\]
The action of $H_{4,2}$ on $Z_4$ descents to a transitive action on
the quotient manifold
\[
  \bar{Z_4}=\{\bar{\phi}(x)\mid x\in Z_4\}.
\]
Therefore to show that the $H_{4,2}$ equivariant action map
\[
  \rho_{\fZ_4}: H_{4,2}\times \fZ_4\rightarrow Z_4
\]
is a surjective submersion, it suffices to show the same for its
restriction map
\[
    (\bar{\phi}\circ\rho_{\fZ_4})^{-1}(\bar{x}_0)\rightarrow
  \bar{\phi}^{-1}(\bar{x}_0).
\]
Denote by $N_{4,2}$ the unipotent radical of $S_{4,2}$. Then
\[
   (N_{4,2}\times N_{4,2})\times \fZ_4\subset
   (\bar{\phi}\circ\rho_{\fZ_4})^{-1}(\bar{x}_0),
\]
and hence it suffices to show that the action map
\begin{equation}\label{submersionx}
   (N_{4,2}\times N_{4,2})\times \fZ_4\rightarrow \bar{\phi}^{-1}(\bar{x}_0)
\end{equation}
is a surjective submersion.

Now let
\[
  x=\left[\begin{array}{cccccc} x_{11}&x_{12}&x_{13}&x_{14}&0&0\\
  x_{21}&x_{22}&x_{23}&x_{24}&0&0\\
  x_{31}&x_{32}&0&0&0&0\\
  x_{41}&x_{42}&0&1&0&0\\
   0&0&0&0&0&1\\
    0&0&0&0&1&0\\
  \end{array}\right]\in \bar{\phi}^{-1}(\bar{x}_0) .
\]
Then $u(x)xv(x)\in \fZ_4$, with
\[
 u(x)=\left[\begin{array}{cccccc} 1&0&0&-x_{14}&0&0\\
  0&1&0&-x_{24}&0&0\\
  0&0&1&0&0&0\\
  0&0&0&1&0&0\\
   0&0&0&0&1&0\\
    0&0&0&0&0&1\\
  \end{array}\right]
\]
and
\[
   v(x)=\left[\begin{array}{cccccc} 1&0&0&0&0&0\\
  0&1&0&0&0&0\\
  0&0&1&0&0&0\\
  -x_{41}&-x_{42}&0&1&0&0\\
   0&0&0&0&1&0\\
    0&0&0&0&0&1\\
  \end{array}\right],
\]
which proves that the map (\ref{submersionx}) is surjective. One
shows similarly that the differential of the map (\ref{submersionx})
is also surjective.
\end{proof}

Let
\[
  \fZ_{4,1}=\{\,\textrm{$x\in \fZ_4$ of the form (\ref{xequals}) with
$x_{13}=x_{31}$}\,\}.
\]

\begin{m4stable}
The closed submanifold $\fZ_{4,1}$ is relatively $H_{4,2}$ stable in
$\fZ_4$.
\end{m4stable}
\begin{proof}
Let $x\in \fZ_{4,1}$,
\[
   g=\left[\begin{array}{ccc} a&b&0\\ 0&a&0\\0&0&a\\
                   \end{array}\right]\in S_{4,2} \quad \textrm{and}\quad  g'=\left[\begin{array}{ccc} a'&0&0\\ b'&a'&0\\0&0&a'\\
                   \end{array}\right]\in S_{4,2}^\tau.
\]
We need to show that $gxg'\in \fZ_{4,1}$, provided that $gxg'\in
\fZ_4$. The condition $gxg'\in \fZ_4$ implies that
\[
   a \left[\begin{array}{cc} 0&1\\ 1&0\\
                   \end{array}\right]a'=\left[\begin{array}{cc} 0&1\\ 1&0\\
                   \end{array}\right]\quad \textrm{and}\quad a \left[\begin{array}{cc} 0&0\\ 0&1\\
                   \end{array}\right]a'=\left[\begin{array}{cc} 0&0\\ 0&1\\
                   \end{array}\right],
\]
which is equivalent to
\[
  a=\alpha \left[\begin{array}{cc} 1&0\\ t&1\\
                   \end{array}\right] \quad \textrm{and}\quad  a'=\alpha^{-1} \left[\begin{array}{cc} 1&-t\\
                   0&1\\
                   \end{array}\right],
\]
for some $\alpha\in \K^\times$ and $t\in \K$. It is now
straightforward to check that $gxg'\in \fZ_{4,1}$.
\end{proof}

\noindent {\bf Remark}: In the sequel, we will skip the verification
when we assert that a submanifold is relatively stable or is a slice
with respect to a certain group action.

\vsp

Write
\[
  Z_{4,1}=H_{4,2} \fZ_{4,1},
\]
which is a closed submanifolds of $Z_4$, by Lemma
\ref{localization1}.

\begin{m42}
\label{m42} The submanifold $Z_4\setminus Z_{4,1}$ is unipotently
$\chi_{4,2}$-incompatible.
\end{m42}
\begin{proof}
Let $x\in \fZ_4\setminus \fZ_{4,1}$ be as in (\ref{xequals}) and
write
\[
  u(x,t)=\left[\begin{array}{cccccc}
    1&0&x_{13}t&0&0&0\\
     0&1&x_{23}t&0&0&0\\
      0&0&1&0&0&0\\
      0&0&0&1&0&0\\
      0&0&0&0&1&0\\
      0&0&0&0&0&1\\
  \end{array}\right]
\]
and
\[
  v(x,t)=
  \left[\begin{array}{cccccc}
    1&0&0&0&0&0\\
     0&1&0&0&0&0\\
      tx_{31}&tx_{32}&1&0&0&0\\
      0&0&0&1&0&0\\
      0&0&0&0&1&0\\
      0&0&0&0&0&1\\
  \end{array}\right].
\]
Then
\[
  u(x,t)x=xv(t,x),  \quad\textrm{i.e.,} \quad (u(x,t),v(x,t)^{-\tau})x=x.
\]
Since
$x_{13}\neq x_{31}$,
\[
  \chi_{4,2}(u(x,t),v(x,t)^{-\tau})= \psi_{\K}(x_{13}t-x_{31}t)\neq 1
\]
for a suitably chosen $t\in \K$. This proves the lemma.
\end{proof}

Write
\[
  \fZ_{4,2}=\{\,\textrm{$x\in \fZ_{4,1}$ of the form (\ref{xequals}) with
$x_{13}=x_{31}=0$}\,\},
\]
which is a relatively $H_{4,2}$ stable closed submanifold of
$\fZ_{4,1}$. Therefore
\[
   Z_{4,2}=H_{4,2} \fZ_{4,2}
\]
is a closed submanifold of $Z_{4,1}$.

\begin{m43}\label{m43}
The submanifold $Z_{4,1}\setminus Z_{4,2}$ is metrically proper in
$G_{4,2}$.
\end{m43}

\begin{proof}
Denote by $\fZ_{4,1}'$ all matrices in $G_{4,2}$ of the form
\begin{equation}\label{x15}
 x=\left[\begin{array}{cccccc} 0&0&a&0&0&0\\
  0&x_{22}&0&0&0&0\\
  a&0&0&0&0&0\\
  0&0&0&1&0&0\\
  0&0&0&0&0&1\\
  0&0&0&0&1&b\\
  \end{array}\right],
  \end{equation}
which forms an $H_{4,2}$ slice of $Z_{4,1}\setminus Z_{4,2}$.

Let $x$ be as in (\ref{x15}). Then one checks that
\[
   \oT_x (Z_{4,1})=\oT_x ({\fZ_{4,1}'})+(\Lie S_{4,2})x+x(\Lie S_{4,2})^\tau\subset
   \gl_4(\K)_{13=31}\times \gl_2(\K),
\]
where $\gl_4(\K)_{13=31}$ is the set of matrices in $\gl_4(\K)$
whose $(1,3)$ entry equals its $(3,1)$ entry. We shall adopt similar
notations in the sequel.

Let
\[
  x'=e_{13}-e_{31}\in \gl_6(\K),
\]
where $e_{ij}$ denotes the matrix with $1$ at the $(i,j)$ entry
and $0$ elsewhere. Then
\[
  \gl_4(\K)_{13=31}\times
  \gl_2(\K)\subset(\K x')^\perp,
  \]
where $\perp$ denotes the orthogonal complement with respect to the
real trace form. Consequently,
\[
  x^{-1}\oT_x (Z_{4,1})\subset x^{-1}(\K x')^\perp =(\K xx')^\perp.
\]
Note that
\[
  xx'=a(e_{11}-e_{33}),
\]
which spans a nondegenerate $\K$ subspace of $\oT_e(G_{4,2})$. This
implies that $x^{-1}\oT_x (Z_{4,1})$ is contained in a proper
nondegenerate subspace of $\oT_e(G_{4,2})$. Therefore by invariance
of the metric, $\oT_x (Z_{4,1})$ is contained in a nondegenerate
proper subspace of $\oT_x(G_{4,2})$, for any $x \in Z_{4,1}\setminus
Z_{4,2}$.
\end{proof}

Denote by ${\fZ_{4,2}}'$ all matrices in $\fZ_{4,2}$ of the form
\begin{equation}\label{x2}
 x=\left[\begin{array}{cccccc}
   x_{11}&0&0&0&0&0\\
  0&0&x_{23}&0&0&0\\
  0&x_{32}&0&0&0&0\\
   0&0&0&1&0&0\\
  0&0&0&0&0&1\\
  0&0&0&0&1&0\\
  \end{array}\right],
  \end{equation}
which also forms an $H_{4,2}$ slice of $Z_{4,2}$.

Write
\[
  {\fZ_{4,3}^{1}}'=\{\,\textrm{$x\in{\fZ_{4,2}}'$ of the form (\ref{x2}) with
$x_{23}=x_{32}$}\,\},
\]
and
\[
  {\fZ_{4,3}^2}'=\{\,\textrm{$x\in {\fZ_{4,2}}'$ of the form (\ref{x2}) with $x_{23}x_{32}+x_{11}=0$}\,\}.
\]

They are both relatively $H_{4,2}$ stable closed submanifolds of
${\fZ_{4,2}}'$. Therefore both
\[
  Z_{4,3}^1=H_{4,2}{\fZ_{4,3}^{1}}'  \quad \textrm{and}\quad  Z_{4,3}^2=H_{4,2} {\fZ_{4,3}^{2}}'
\]
are closed submanifolds of $Z_{4,2}$.

\begin{m44}\label{m44}
The manifold $Z_{4,2}\setminus(Z_{4,3}^1\cup Z_{4,3}^2)$ is
unipotently $\chi_{4,2}$-incompatible.
\end{m44}
\begin{proof}
Let $x\in {\fZ_{4,2}}'\setminus({\fZ_{4,3}^{1}}'\cup
{\fZ_{4,3}^{2}}')$ be as in (\ref{x2}). Set
\[
 u(x,t)=\left[\begin{array}{cccccc}
  1&0&x_{11}x_{32}^{-1}t&0&0&0\\
  t&1&0&x_{23}t&0&0\\
  0&0&1&0&0&0\\
  0&0&t&1&0&0\\
  0&0&0&0&1&0\\
  0&0&0&0&t&1\\
  \end{array}\right],
 \]
and
\[
 v(x,t)=\left[\begin{array}{cccccc}
  1&t&0&0&0&0\\
  0&1&0&0&0&0\\
  x_{11}x_{23}^{-1}t&0&1&t&0&0\\
  0&x_{32}t&0&1&0&0\\
  0&0&0&0&1&t\\
  0&0&0&0&0&1\\
  \end{array}\right].
 \]
Then $u(x,t)x=xv(x,t)$ and
\[\chi _{4,2}(u(x,t), v(x,t)^{-\tau})=\psi
_{\K}((x^{-1}_{32}-x^{-1}_{23})t(x_{11}+x_{23}x_{32}))\ne 1 \] for a
suitably chosen $t$. The lemma follows.
\end{proof}

\begin{m45}\label{m45}
The submanifold $Z_{4,3}^1$ is metrically proper in $G_{4,2}$.
\end{m45}
\begin{proof}
Let $x\in {\fZ_{4,3}^1}'$ as in (\ref{x2}). Write $a=x_{23}=x_{32}$
and
\[
  x'=\left[\begin{array}{cccccc}
  0&0&0&0&0&0\\
  0&0&1&0&0&0\\
  0&-1&0&0&0&0\\
   0&0&0&0&0&0\\
   0&0&0&0&0&a\\
   0&0&0&0&-a&0\\
  \end{array}\right].
  \]
Then one checks that
\[
   \oT_x ({Z_{4,3}^1})=\oT_x ({\fZ_{4,3}^1}')+(\Lie S_{4,2})x+x(\Lie S_{4,2})^\tau\subset
   (\K x')^\perp.
\]
Note that
\[
 x'x=a (\diag(0,1,-1,0,1,-1)),
  \]
which spans a nondegenerate $\K$ subspace of $\oT_e(G_{4,2})$. Here
and as usual, $\diag$ represents a diagonal matrix (with the obvious
diagonal entries). The lemma follows, as in the proof of Lemma
\ref{m43}.
\end{proof}

\begin{m46}\label{m46}
The submanifold $Z_{4,3}^2$ is metrically proper in $G_{4,2}$.
\end{m46}
\begin{proof}
Let $x\in {\fZ_{4,3}^2}'$ as in (\ref{x2}). Write
\[
  x'=\left[\begin{array}{cccccc}
  1&0&0&0&0&0\\
  0&0&x_{23}&0&0&0\\
  0&x_{32}&0&0&0&0\\
  0&0&0&-x_{23}x_{32}&0&0\\
  0&0&0&0&0&0\\
  0&0&0&0&0&0\\
    \end{array}\right]
  \]
Then one checks that
\[
   \oT_x ({Z_{4,3}^2})=\oT_x ({\fZ_{4,3}^2}')+(\Lie S_{4,2})x+x(\Lie S_{4,2})^\tau\subset
   (\K x')^\perp,
\]
and
\[
 x'x=x_{23}x_{32}(\diag(-1,1,1,-1,0,0)).
  \]
The lemma follows, as before.
\end{proof}

\vsp We now consider the $H_{4,2}$ stable filtration
\[
  Z_4\supset Z_{4,1}\supset Z_{4,2}\supset Z_{4,3}^1\cup
  Z_{4,3}^2\supset Z_{4,3}^1\supset \emptyset .
\]
In view of the proceeding lemmas, the proof of Proposition \ref{m4}
is complete.

\section{A submanifold $Z_6$ of $\GL_6$}
\label{gl6}

Recall from Section \ref{smallsub} the $P-P^{\tau}$ double coset
$G_R$ indexed by a rank matrix $R$. Set
\begin{equation}
 Z_6=G_R, \quad\textrm{with } R=\left[\begin{array}{cc}
    3&2\\
    2&1\\
  \end{array}\right].
\end{equation}
Clearly $Z_6$ is an $H=S\times S$ stable submanifold of $G$, as with
each $G_R$.

The purpose of this section is to prove the following proposition.
Again it will take a number of steps.

\begin{m6}\label{m6}
As an $H$ submanifold of $G$, $Z_6$ has $\oU_{\chi}\oM$ property.
\end{m6}

Denote by $\fZ_6$ all matrices in $Z_6$ of the form
\begin{equation}\label{z6}
 x=\left[\begin{array}{cccccc}
  *&*&*&*&x_{15}&0\\
  *&*&*&*&x_{25}&0\\
 *&*&0&0&x_{35}&0\\
 *&*&0&0&x_{45}&0\\
 x_{51}&x_{52}&x_{53}&x_{54}&0&0\\
  0&0&0&0&0&1\\
  \end{array}
  \right],
\end{equation}
which forms an $H$ slice of $Z_6$. Write
\[
  \fZ_{6,1}=\{\,\textrm{$x\in \fZ_6$ of the form (\ref{z6}) with
$x_{35}=x_{53}$}\,\},
\]
and
\[
  \fZ_{6,2}=\{\,\textrm{$x\in \fZ_6$ of the form (\ref{z6}) with
$x_{35}=x_{53}=0$}\,\}.
\]
They are both relatively $H$ stable closed submanifolds of $\fZ_6$.
Therefore both
\[
  Z_{6,1}=H \fZ_{6,1}  \quad \textrm{and } Z_{6,2}=H \fZ_{6,2}
\]
are closed submanifolds of $Z_6$.

\begin{m62}
The submanifold $Z_6\setminus Z_{6,1}$ is unipotently
$\chi$-incompatible.
\end{m62}
\begin{proof}
Let $x\in \fZ_6\setminus \fZ_{6,1}$ be as in (\ref{z6}). Write
\[
u(x,t)=\left[\begin{array}{cccccc}
  1&0&0&0&x_{15}t&0\\
  0&1&0&0&x_{25}t&0\\
 0&0&1&0&x_{35}t&0\\
 0&0&0&1&x_{45}t&0\\
 0&0&0&0&1&0\\
  0&0&0&0&0&1\\
  \end{array}
  \right],
\]
and
\[
  v(x,t)=\left[\begin{array}{cccccc}
  1&0&0&0&0&0\\
  0&1&0&0&0&0\\
 0&0&1&0&0&0\\
 0&0&0&1&0&0\\
  tx_{51}&tx_{52}&tx_{53}&tx_{54}&1&0\\
  0&0&0&0&0&1\\
  \end{array}
  \right].
\]
Then $u(x,t)x=xv(x,t)$, and the lemma follows, as before.
\end{proof}

\begin{m63}
The submanifold $Z_{6,1}\setminus Z_{6,2}$ is metrically proper in
$G$.
\end{m63}
\begin{proof}
Every element of $\fZ_{6,1}\setminus \fZ_{6,2}$ is in the same
$H$-orbit as an element of the form
\begin{equation*}\label{xg1}
 x=\left[\begin{array}{cccccc}
  *&*&0&*&0&0\\
  *&*&0&*&0&0\\
 0&0&0&0&a&0\\
 *&*&0&0&0&0\\
 0&0&a&0&0&0\\
  0&0&0&0&0&1\\
  \end{array}
  \right],
\end{equation*}
Fix such an $x$. Then
\[
 \oT_x(Z_{6,1})=\oT_x(\fZ_{6,1})+\Lie(S)x+x\Lie (S^\tau)\subset
 \gl_6(\K)_{35=53}=(\K x')^\perp,
\]
where $x'=e_{35}-e_{53}$. Now $x'x=a(e_{33}-e_{55})$ and we finish
the proof, as before.
\end{proof}

Denote by $\fZ_{6,2}'$ all matrices in $\fZ_{6,2}$ of the form
\begin{equation}\label{z62p}
 x=\left[\begin{array}{cccccc}
  x_{11}&x_{12}&x_{13}&0&0&0\\
  x_{21}&x_{22}&x_{23}&0&0&0\\
  x_{31}&x_{32}&0&0&0&0\\
 0&0&0&0&1&0\\
 0&0&0&1&0&0\\
  0&0&0&0&0&1\\
  \end{array}
  \right],
\end{equation}
which also forms an $H$ slice of $Z_{6,2}$. Set
\[
  \fZ_{6,3}'=\{\,\textrm{$x\in \fZ_{6,2}'$ of the form (\ref{z62p}) with
$x_{13}=x_{31}$}\,\},
\]
and
\[
  \fZ_{6,4}'=\{\,\textrm{$x\in \fZ_{6,2}'$ of the form (\ref{z62p}) with
$x_{13}=x_{31}=0$}\,\}.
\]
They are both relatively $H$ stable closed submanifolds of
$\fZ_{6,2}'$. Therefore both
\[
  Z_{6,3}=H \fZ_{6,3}'  \quad \textrm{and } Z_{6,4}=H \fZ_{6,4}'
\]
are closed submanifolds of $Z_{6,2}$.

\begin{m64}
The manifold $Z_{6,2}\setminus Z_{6,3}$ is unipotently
$\chi$-incompatible.
\end{m64}
\begin{proof}
This is identical to the proof of Lemma \ref{m42} in Section
\ref{sm4}. We omit the details.
\end{proof}

\begin{m65}
The submanifold $Z_{6,3}\setminus Z_{6,4}$ is metrically proper in
$G$.
\end{m65}
\begin{proof}
Every matrix in $\fZ_{6,3}'\setminus \fZ_{6,4}'$ is in the same $H$
orbit as a matrix of the form
\begin{equation*}\label{z62pp}
 x=\left[\begin{array}{cccccc}
 0&0&a&0&0&0\\
  0&x_{22}&0&0&0&0\\
  a&0&0&0&0&0\\
 0&0&0&0&1&0\\
 0&0&0&1&0&0\\
  0&0&0&0&0&1\\
  \end{array}
  \right].
\end{equation*}
Fix such an $x$. Then
\[
 \oT_x(Z_{6,3})=\oT_x(\fZ_{6,3}')+\Lie(S)x+x\Lie (S^\tau) \subset
 \gl_6(\K)_{13=31}=(\K x')^\perp,
\]
where $x'=e_{13}-e_{31}$. Now $x'x=a(e_{11}-e_{33})$ and we finish
the proof, as before.
\end{proof}

Denote by $\fZ_{6,4}''$ all matrices in $\fZ_{6,4}'$ of the form
\begin{equation}\label{x64pp}
 x=\left[\begin{array}{cccccc}
  x_{11}&0&0&0&0&0\\
  0&0&x_{23}&0&0&0\\
  0&x_{32}&0&0&0&0\\
 0&0&0&0&1&0\\
 0&0&0&1&0&0\\
  0&0&0&0&0&1\\
  \end{array}
  \right],
\end{equation}
which also forms an $H$ slice of $Z_{6,4}$. Set
\[
  \fZ_{6,5}''=\{\,\textrm{$x\in \fZ_{6,4}''$ of the form (\ref{x64pp}) with
$x_{23}=x_{32}$}\,\},
\]
which is a relatively $H$ stable closed submanifolds of
$\fZ_{6,4}''$. Therefore
\[
  Z_{6,5}=H \fZ_{6,5}''
\]
is a closed submanifold of $Z_{6,4}$.

\begin{m66}
The manifold $Z_{6,4}\setminus Z_{6,5}$ is unipotently
$\chi$-incompatible.
\end{m66}
\begin{proof}
Let $x\in \fZ_{6,4}''\setminus \fZ_{6,5}''$ be as in (\ref{x64pp}).
Set
\[
u(x,t)=\left[\begin{array}{cccccc}
  1&0&x_{11}x_{32}^{-1}t&0&0&0\\
  t&1&0& 0&x_{23}t&0\\
  0&0&1&0&0&0\\
 0&0&t&1&0&t\\
 0&0&0&0&1&0\\
  0&0&0&0&t&1\\
  \end{array}
  \right],
\]
and
\[
v(x,t)=\left[\begin{array}{cccccc}
  1&t&0&0&0&0\\
   0&1&0&0&0&0\\
  x_{23}^{-1}tx_{11}&0&1& t&0&0\\
   0&0&0&1&0&0\\
 0&tx_{32}&0&0&1&t\\
  0&0&0&t&0&1\\
  \end{array}
  \right].
\]
Then
\[
  u(x,t)x=\left[\begin{array}{cccccc}
   x_{11}  &x_{11}t &0      &0           &0   &0\\
   tx_{11} &0       &x_{23} &x_{23}t     &0   &0\\
   0       &x_{32}  &0      &0           &0   &0\\
   0       &tx_{32} &0      &0           &1   &t\\
   0       &0       &0      &1           &0   &0\\
   0       &0       &0      &t           &0   &1\\
  \end{array}
  \right]=xv(x,t)
\]
and the lemma follows, as before.
\end{proof}

\begin{m67}
The submanifold $Z_{6,5}$ is metrically proper in $G$.
\end{m67}

\begin{proof}
Let $x\in \fZ_{6,5}''$ be as in (\ref{x64pp}), with
$x_{23}=x_{32}=a$. Write
\begin{equation}\label{xa2}
 x'=\left[\begin{array}{cccccc}
  0&0&0&0&0&0\\
  0&0&1&0&0&0\\
  0&-1&0&0&0&0\\
 0&0&0&0&-a&0\\
 0&0&0&a&0&0\\
  0&0&0&0&0&0\\
  \end{array}
  \right].
\end{equation}
Then
\[
  \oT_x(Z_{6,5})=\oT_x(Z_{6,5})+\Lie(S)x+x\Lie (S^\tau)\subset
  (\K x')^\perp,
\]
and $x'x=a(\diag(0,1,-1,-1,1,0))$. The lemma follows, as before.
\end{proof}

\vsp We now consider the $H$ stable filtration
\[
  Z_{6}\supset Z_{6,1}\supset Z_{6,2}\supset Z_{6,3}\supset
  Z_{6,4}\supset Z_{6,5}\supset \emptyset .
\]
In view of the proceeding lemmas, the proof of Proposition \ref{m6}
is complete.

\vsp

\section{The manifold $\GL_2\times \GL_2$}
\label{gl222}

\vsp Set
\[
  M_2=\GL_2(\K)\times\GL_2(\K) =\GL_2(\K)\times\GL_2(\K)\times \{I_2\}\subset G,
\]
which is stable under the subgroup
\[
  H_2=\GL_2(\K)=\{(x, x^{-\tau})\mid x\in \GL_2^{\Delta}(\K)\}\subset H =S\times S.
\]

It will be slightly more convenient to work with the following:
\[
  \tilde{H}_{2}=\{1,\tau\}\ltimes \GL_2(\K),
\]
where the semidirect product is given by the action
\[
  \tau(g)=g^{-\tau}.
\]
Denote by $\tilde{\chi}_{2}$ the character of $\tilde{H}_{2}$ such
that
\[
  \tilde{\chi}_{2}|_{\GL_2(\K)}=1\quad\textrm{and}\quad \tilde{\chi}_{2}(\tau)=-1.
\]

\begin{m2}\label{m2}
Let $\tilde{H}_{2}$ act on $M_2=\GL_2(\K)\times \GL_2(\K)$ by
\[
  g(x,y)=(gxg^{-1},gyg^{-1}),\quad g\in \GL_2(\K),
\]
and
\[
  \tau(x,y)=(x^\tau,y^\tau).
\]
Then
\[
  \con^{-\xi}_{\tilde{\chi}_{2}}(M_2)=0.
\]
\end{m2}
\begin{proof}
Using the same formula, we may extend the action of
$\tilde{H}_{2}$ on $\GL_2(\K)\times \GL_2(\K)$ to the larger space
$\gl_2(\K)\times \gl_2(\K)$. By Lemma \ref{nash}, it suffices to
prove that
\[
   \con^{-\xi}_{\tilde{\chi}_{2}}(\gl_2(\K)\times \gl_2(\K))=0.
\]

Identify $\K$ with the center of $\gl_2(\K)$. We have
\[
   \gl_2(\K)\times \gl_2(\K)=(\s\l_2(\K)\times \s\l_2(\K))\oplus
   (\K\times \K)
\]
as a $\K$ linear representation of $\tilde{H}_{2}$, where
$\tilde{H}_{2}$ acts on $\K\times \K$ trivially. Therefore it
suffices to prove that
\[
   \con^{-\xi}_{\tilde{\chi}_{2}}(\s\l_2(\K)\times \s\l_2(\K))=0.
\]
We view $\sl_2(\K)$ as a three dimensional quadratic space under
the trace form. Under this identification, the action of
$\tilde{H}_{2}$ yields the diagonal action of $\oO(\sl_2(\K))$ on
$\sl_2(\K)\times \sl_2(\K)$, with $\tilde{\chi}_{2}$ corresponding
to the determinant character. So the required vanishing result is
a special case of Proposition \ref{first-occur}.
\end{proof}

\vsp

\section{The manifold $\GL_3\times \GL_1$}\label{sgl3}

Set \[
  M_3=\left\{\left[\begin{array}{cccccc}
                   *&*&*&0&0&0\\
                   *&*&*&0&0&0\\
                   *&*&x_{33}&0&0&0\\
                   0&0&0&x_{44}&0&0\\
                   0&0&0&0&1&0\\
                   0&0&0&0&0&1\\
              \end{array}\right]\in G \mid x_{33}\neq
              x_{44}\right\},
\]
which is stable under the subgroup $H_3$ of $H=S\times S$
consisting of elements of the form
\[
  \left(\left[\begin{array}{cccccc}
                   a&0&*&0&0&0\\
                   0&1&*&0&0&0\\
                   0&0&a&0&0&0\\
                   0&0&0&1&0&0\\
                   0&0&0&0&a&0\\
                   0&0&0&0&0&1\\
              \end{array}\right],
            \left[\begin{array}{cccccc}
                   a^{-1}&0&*&0&0&0\\
                   0&1&*&0&0&0\\
                   0&0&a^{-1}&0&0&0\\
                   0&0&0&1&0&0\\
                   0&0&0&0&a^{-1}&0\\
                   0&0&0&0&0&1\\
              \end{array}\right]
            \right).
\]

Write
\[
  L_3=\left\{\,\left[\begin{array}{ccc} a&0&0\\ 0&1&0\\0&0&a\\
                   \end{array}\right]
 \mid  a\in \K^\times\right\}
\]
and
\[
 N_3=\left\{\,\left[\begin{array}{ccc} 1&0&d\\ 0&1&c\\0&0&1\\
                   \end{array}\right]
 \mid c,d\in\K \right\}.
\]
Then
\[
   H_3=L_3\ltimes (N_3\times N_3),
\]
where the semidirect product is defined by the action
\[
  l(g_1,g_2)=(l g_1 l^{-1},l^{-1}g_2 l).
\]
Define
\[
    \tilde{H_3}=\{1,\tau\}\ltimes H_3,
 \]
with the semidirect product given by the action
\[
  \tau (l,g_1,g_2)=(l^{-1}, g_2, g_1).
\]
Write
\[
   \chi_{N_3} \left(\,\left[\begin{array}{ccc} 1&0&d\\ 0&1&c\\0&0&1\\
                   \end{array}\right]
                   \right)=\psi_{\K}(d),
\]
and let $\tilde{\chi}_3$ be the character of $\tilde{H}_3$ such
that
\[
  \tilde{\chi}_3(l,g_1,g_2)=\chi_{N_3}(g_1)\chi_{N_3}(g_2),\quad
  (l,g_1,g_2)\in H_3,
\]
and
\[
  \tilde{\chi}_3(\tau)=-1.
\]

\begin{m3}\label{m3}
Let $\tilde{H}_3$ act on
\begin{equation} \label{dm3}
  M_3=\{(x,y)\in\GL_3(\K)\times \K^\times \mid y\neq \textrm{the (3,3) entry of $x$}\}
\end{equation}
by
\[
  (l,g_1,g_2)(x,y)=(lg_1 x g_2^\tau l^{-1},y)
\]
and
\[
  \tau (x,y)=(x^\tau,y).
\]
Then
\[
  \con^{-\xi}_{\tilde{\chi}_3}(M_3)=0.
\]
\end{m3}

\begin{proof} First we note that the (3, 3) entry of $x$ is
invariant under $\tilde{H}_3$. Denote by $\GL_3(\K)'$ the set of
matrices in $\GL_3(\K)$ whose (3, 3) entry is not 1. Let
$\tilde{H}_3$ act on $\GL_3(\K)'\times \K^\times$ by the same
formula as its action on $M_3$. Then the map
\[
   \begin{array}{rcl}
      \GL_3(\K)'\times \K^\times& \rightarrow & M_3,\\
          (x,y)&\mapsto& (yx,y)
   \end{array}
\]
is an $\tilde{H}_3$-equivariant Nash diffeomorphism. Therefore
\[
  \con^{-\xi}_{\tilde{\chi}_3}(M_3)\cong \con^{-\xi}_{\tilde{\chi}_3}( \GL_3(\K)'\times
  \K^\times).
\]
As the action of $\tilde{H}_3$ on $\K^\times$ is trivial, it
suffices to show that
\[
  \con^{-\xi}_{\tilde{\chi}_3}( \GL_3(\K)')=0.
\]
This will be implied by Lemma \ref{nash} and Proposition \ref{m32}
below.
\end{proof}

The rest of this section is devoted to the proof of

\begin{m32}\label{m32}
Let $\tilde{H}_3$ act on $\gl_3(\K)$ by
\[
  (l,g_1,g_2)x=lg_1 x g_2^\tau l^{-1}
\]
and
\[
  \tau x=x^\tau.
\]
Then
\[
  \con^{-\xi}_{\tilde{\chi}_3}(\gl_3(\K))=0.
\]
\end{m32}

\vsp

Write
\[
   Z_{3,1}=\left\{\,\left[\begin{array}{ccc} *&*&*\\ *&*&*\\ *&*&0\\
                   \end{array}\right]\in \gl_3(\K)
           \right\}
\]
and
\[
   Z_{3,2}=\left\{\,\left[\begin{array}{ccc} *&*&a\\ *&*&*\\a&*&0\\
                   \end{array}\right]\in \gl_3(\K)
           \right\}.
\]

\begin{nz1}\label{nz1}
One has that $\con^{-\xi}_{\tilde{\chi}_3}(\gl_3(\K)\setminus
Z_{3,1})=0.$
\end{nz1}
\begin{proof}
The Nash submanifold $\gl_2(\K)\times \K^\times$ is an
$\tilde{H}_3$-slice of $\gl_3(\K)\setminus Z_{3,1}$, which is stable
under the subgroup
\[
  \tilde{H}_{3,1}=\{1,\tau\}\ltimes L_3=\{1,\tau\}\ltimes \K^\times.
\]
Denote by $\tilde{\chi}_{3,1}$ the restriction of $\tilde{\chi}_3$
to $\tilde{H}_{3,1}$. Then we have the injective restriction map
\[
  \con^{-\xi}_{\tilde{\chi}_3}(\gl_3(\K)\setminus Z_{3,1})
  \hookrightarrow\con^{-\xi}_{\tilde{\chi}_{3,1}}(\gl_2(\K)\times
  \K^\times).
\]

Let $\tilde{H}_{3,1}$ act on $\K_3=\K\times \K\times \K^\times$
trivially, and act on $\K\times \K$ by
\[
  a(x,y)=(ax,a^{-1}y), \,\,a\in\K^\times\quad \textrm{and}\quad \tau(x,y)=(y,x).
\]
Then
\[
  \gl_2(\K)\times
  \K^\times\cong (\K\times \K)\times \K_3
\]
as Nash manifolds with $\tilde{H}_{3,1}$ actions. It thus suffices
to show that
\begin{equation}\label{vanishkk}
  \con^{-\xi}_{\tilde{\chi}_{3,1}}(\K\times \K)=0.
\end{equation}
We view $\K\times \K$ as a split two dimensional quadratic space
so that both $\K\times \{0\}$ and $\{0\}\times \K$ are isotropic.
Then $\tilde{H}_{3,1}$ is identified with the orthogonal group
$\oO(\K\times \K)$, with $\tilde{\chi}_{3,1}$ corresponding to the
determinant character. So (\ref{vanishkk}) is a special case of
Proposition \ref{first-occur}.
\end{proof}

\begin{gl3z1}\label{gl3z1}
The $H_3$ stable manifold $Z_{3,1}\setminus Z_{3,2}$ is unipotently
$\chi_3$-incompatible, where $\chi _3=\tilde{\chi}_3|_{H_3}$.
\end{gl3z1}
\begin{proof}
For
\[
  x=\left[\begin{array}{ccc}
    x_{11}&x_{12}&x_{13}\\
     x_{21}&x_{22}&x_{23}\\
      x_{31}&x_{32}&0\\
  \end{array}\right]\in Z_{3,1}\setminus Z_{3,2} \quad
  \textrm{and}\quad t\in \K,
\]
write
\[
  u(x,t)=\left[\begin{array}{ccc}
    1&0&x_{13}t\\
     0&1&x_{23}t\\
      0&0&1\\
  \end{array}\right]\quad \textrm{and}\quad
  v(x,t)=
  \left[\begin{array}{ccc}
    1&0&0\\
     0&1&0\\
      tx_{31}&tx_{32}&1\\
  \end{array}\right].
\]
Then
\[
  u(x,t)x=xv(t,x),
\]
and the lemma follows, as in the proof of Lemma \ref{m42}.
\end{proof}

Lemma \ref{nz1}, Lemma \ref{gl3z1} and Lemma \ref{localization3}
now imply the following

\begin{gl3z0}\label{gl3z0}
Every generalized function in
$\con^{-\xi}_{\tilde{\chi}_3}(\gl_3(\K))$ is supported in
$Z_{3,2}$.
\end{gl3z0}

\vsp We shall employ Fourier transform to finish the proof of
Proposition \ref{m32}. In general, let $E$ be a finite dimensional
real vector space, equipped with a nondegenerate symmetric
bilinear form $\la\,,\,\ra_E$. The Fourier transform is a
topological linear isomorphism
\[
  \widehat{\phantom{f}}:\mathcal{S}(E)\rightarrow
\mathcal{S}(E)
\]
of the space of Schwartz functions, given by
\[
  \widehat{f}(x)=\int_E f(y)e^{-2\pi \sqrt{-1}\,\la x,y\ra_E}\,dy,
\]
where $dy$ is the Lebesgue measure on $E$, normalized such that
the volume of the cube
\[
  \{t_1v_1+t_2v_2+\cdots +t_rv_r\mid 0\leq t_1,t_2,\cdots, t_r\leq 1\}
\]
is $1$, for any orthogonal basis $v_1,v_2,\cdots,v_r$ of $E$ such
that $\la v_i,v_i\ra_E=\pm 1$, $i=1,2,\cdots, r$. The Fourier
transform extends continuously to a topological linear isomorphism
\[
  \widehat{\phantom{f}}:\con^{-\xi}(E)\rightarrow
\con^{-\xi}(E),
\]
which is still called the Fourier transform.

The following lemma is a form of uncertainty principle.

\begin{fourier}\label{fourier}
Let $f\in \con^{-\xi}(E)$. If both $f$ and $\widehat{f}$ are
supported in a common nondegenerate proper subspace of $E$, then
$f=0$.
\end{fourier}

\begin{proof}
Let $v\in E$ be a nondegenerate vector such that both $f$ and
$\widehat{f}$ are supported in its perpendicular space. Denote by
$v^*$ the function
\[
  E\rightarrow \R,\quad u\mapsto \la u,v\ra_E.
\]
Due to tempered-ness, $\widehat{f}$ has a finite order and
therefore
\[
   (v^*)^k \,\widehat{f} = 0 \quad \textrm{for some }k\geq 1.
\]
Consequently $(\partial/\partial v)^k \,f=0$, and we finish the
proof by applying Lemma \ref{tvanishing2}.
\end{proof}

We continue with the proof of Proposition \ref{m32}. Let $\gl_3(\K)$
be equipped with the real trace from as in the Introduction and
define the Fourier transform accordingly. Given $f\in
\con^{-\xi}_{\tilde{\chi}_3}(\gl_3(\K))$, it is easy to check that
its Fourier transform $\widehat{f}\in \con^{-\xi}(\gl_3(\K))$
satisfies the followings:
\[
  \left\{
    \begin{array}{l}
           \textrm{(a) } \widehat{f}(lxl^{-1})=\widehat{f}(x), \quad l\in L_3,\smallskip\\
            \textrm{(b) } \widehat{f}(g_1^\tau x g_2)=\chi_{N_3}(g_1)^{-1} \chi_{N_3}(g_2)^{-1} \widehat{f}(x),
             \quad g_1, g_2\in N_3,\,\,\textrm{and,} \smallskip\\
             \textrm{(c) } \widehat{f}(x^\tau)=-\widehat{f}(x).
    \end{array}
  \right.
\]
Then as in Lemma \ref{gl3z0}, we conclude that $\widehat{f}$ is
supported in
\[
  Z'_{3,2}=\left\{\,\left[\begin{array}{ccc} 0&*&a\\ *&*&*\\a&*&*\\
                   \end{array}\right]\in \gl_3(\K)
           \right\}.
\]
Therefore both $f$ and $\widehat f$ are supported in the proper
nondegenerate subspace
\[
Z_{3,2}+ Z'_{3,2}=
\left\{\,\left[\begin{array}{ccc} *&*&a\\ *&*&*\\a&*&*\\
                   \end{array}\right]\in \gl_3(\K)
           \right\}.
\]
Lemma \ref{fourier} then implies that $f=0$. The proof of
Proposition \ref{m32} is now complete.

\vsp \noindent {\bf Remark}: We may view the Fourier transform
argument of this section as a variation of the metrical properness
argument of Sections \ref{sm4} and \ref{gl6}. In view of Lemma
\ref{gl3z1} on unipotent $\chi_3$-incompatibility, we have in some
sense used $\oU_\chi \oM$ property to reduce Proposition \ref{m3} to
the vanishing of (\ref{vanishkk}). The latter is closely related to
the multiplicity one property of the pair $(\GL_2(\K),\GL_1(\K))$.

\section{Proof of Theorem \ref{thm:mainB}}
\label{theoremB}

We will first examine the case where the quaternion algebra $\BD$ is
split, namely $G=\GL_6(\K)$. We start with the following

\begin{um-transitivity}\label{um-transitivity}
Recall the notations of Section \ref{sm4}.
\begin{itemize}
\item[(a)]
If $\fZ$ is a unipotently $\chi_{4,2}$-incompatible $H_{4,2}$ stable
submanifold of $G_{4,2}$, then $Z=H\fZ$ is a unipotently
$\chi$-incompatible submanifold of $G$. \item[(b)] If $\fZ$ is a
metrically proper $H_{4,2}$ stable submanifold of $G_{4,2}$, then
$Z=H\fZ$ is a metrically proper submanifold of $G$.
\end{itemize}
\end{um-transitivity}
\begin{proof} Part (a) is clear. For Part (b), we note \[
  Z=H \fZ=U_{4,2} \fZ \,U_{4,2}^\tau,
\]
where
\[
  U_{4,2}=\left\{\left[\begin{array}{ccc}
    I_2&0&d\\
    0&I_2&c\\
    0&0&I_2\\
     \end{array}\right]\mid c,d\in \gl_2(\K)\right \}.
\]
By invariance of the metric, we only need to show that $Z$ is
metrically proper at every point $z\in \fZ$, i.e., the tangent space
$\oT_z(Z)$ is contained in a nondegenerate proper subspace of
$\oT_z(G)$.

First we assume that $z$ is the identity matrix $e$. Then
\[
  \oT_e(Z)=\oT_e(\fZ)+(\Lie(U_{4,2})+\Lie(U_{4,2}^\tau))
\]
is metrically proper since
\[
  \oT_e(\fZ) \textrm{ is metrically proper in }
  \oT_e(\GL_4(\K)\times
\GL_2(\K)),
  \]
and
\[
  \oT_e(G)=\oT_e(\GL_4(\K)\times
\GL_2(\K)) \oplus (\Lie(U_{4,2})+\Lie(U_{4,2}^\tau))
\]
is an orthogonal decomposition.

Now let $z\in\fZ$. Note that
\[
  z^{-1}Z=U_{4,2}(z^{-1}\fZ) U_{4,2}^\tau,
\]
and
\[
  z^{-1}\fZ\textrm{ is metrically proper in }\GL_4(\K)\times
\GL_2(\K).
\]
Therefore the above argument implies that $z^{-1}Z$ is metrically
proper at $e$. Using the left multiplication by $z$
\[
  l_z: (G, z^{-1}Z, e)\rightarrow (G, Z,z),
\]
we conclude that $Z$ is metrically proper at $z$.
\end{proof}

\vsp

Recall the open submanifold $G'$ of $G$ from Section \ref{smallsub}.
Set \[
  G_{4,2}'=(\GL_4\times \GL_2)\cap G',\]
which is stable under $H_{4,2}=S_{4,2}\times S_{4,2}$. Define
$G_{2,4}'$ and $H_{2,4}$ similarly.

Recall also the submanifolds $M_2$ and $M_3$, from Sections
\ref{gl222} and \ref{sgl3}. Also define the following symmetric
counterpart of $M_3$:
\[
  \check{M_3}=\left\{\left[\begin{array}{cccccc}
                    1&0&0&0&0&0\\
                   0&1&0&0&0&0\\
                   0&0&*&*&*&0\\
                   0&0&*&*&*&0\\
                   0&0&*&*&y_{33}&0\\
                   0&0&0&0&0&y_{44}\\
              \end{array}\right]\in G\mid y_{33}\neq y_{44}\right\}.
\]
Note that
\[M_2 \subset
G'_{4,2}\cap G'_{2,4},\ \ M_3 \subset G'_{4,2}, \ \ \check{M_3}
\subset G'_{2,4}.
\]

We have
\[G_{4,2}'\setminus (H_{4,2}M_2\cup H_{4,2}M_3)= Z_4,\]
\[G_{2,4}'\setminus (H_{2,4}M_2\cup H_{2,4}\check{M_3})= W_4,\]
where $Z_4$ is given in (\ref{z4}), and $W_4$ is given similarly by
\[
   W_4=\left\{\,\left[\begin{array}{ccc}y&0&0\\
    0& a_{11}&a_{12}\\
    0&a_{21}&a_{22}\\
         \end{array}\right]\in G_{2,4} \mid
   \textrm{$y^{-1}a_{22}$ is nilpotent and nonzero}\right\}.
\]

\vsp Let
\begin{equation}
\label{G''} G''=HM_2\cup HM_3\cup H\check{M_3}\subset G'.
\end{equation}

\begin{transitivity}
\label{transitivity}
As an $H$ manifold, $G'\setminus G''$ has $\oU_{\chi}\oM$
property. Consequently if $f\in \con^{-\infty}_{\chi}(G')$ is an
eigenvector of $\Delta$, and $f$ vanishes on $G''$, then $f=0$.
\end{transitivity}
\begin{proof}
It is easy to check that
\begin{itemize}
\item $G'\setminus G'' =Z_6\bigsqcup HZ_4\bigsqcup HW_4$;
\item $Z_6$ is closed in $G'\setminus G''$;
\item Both $HZ_4$ and $HW_4$ are closed in $HZ_4\bigsqcup HW_4$.
\end{itemize}
By Proposition \ref{m6}, the submanifold $Z_6$ has $\oU_{\chi}\oM$
property. By Proposition \ref{m4} and Lemma \ref{um-transitivity},
the submanifold $H Z_4$ has $\oU_{\chi}\oM$ property. Similarly,
$HW_4$ also has $\oU_{\chi}\oM$ property. Therefore the $H$ stable
closed subset $Z_6\bigsqcup HZ_4\bigsqcup HW_4$ of $G'$ has
$\oU_{\chi}\oM$ property. The assertion follows.
\end{proof}

\vsp Now set
\[
  \tilde{H}=\{1,\tau\}\ltimes H=\{1,\tau\}\ltimes (S\times S),
\]
where the semidirect product is defined by the action
\[
  \tau (g_1,g_2)=(g_2, g_1), \quad g_1,g_2\in S.
\]
Extend $\chi$ to a character $\tilde{\chi}$ of $\tilde{H}$ by
requiring
\[
  \cover{\chi}(\tau)=-1,
\]
and extend the action on $G$ of $H$ to $\tilde{H}$ by requiring
\[
  \tau x=x^\tau.
 \]

\begin{symmetry}\label{symmetry}
One has that $\con^{-\xi}_{\tilde{\chi}}(G'')=0$.
\end{symmetry}
\begin{proof}
By using the restriction map, Proposition \ref{m2} implies that
\[
   \con^{-\xi}_{\tilde{\chi}}(H M_2)=0.
\]
Similarly, Proposition \ref{m3} imply that
\[
   \con^{-\xi}_{\tilde{\chi}}(H M_3)=0,
\]
and likewise,
\[
   \con^{-\xi}_{\tilde{\chi}}(H \check{M_3})=0.
\]
The proposition follows from the above three vanishing results.
\end{proof}

\vsp We are now ready to prove Theorem \ref{thm:mainB} for the split
case. Let $f$ be as in the theorem. Write
\[
  f^\tau(x)=f(x^\tau).
\]
Then $f^\tau$ still satisfies (\ref{fsx}), which implies that
\[
  f-f^\tau\in \con^{-\xi}_{\tilde{\chi}}(G).
\]

From Proposition \ref{symmetry}, we know that $f-f^{\tau}=0$ on
$G''$. Note that $\tau$ commutes with the differential operator
$\Delta$ on $G$. So $f^\tau$ is an eigenvector of $\Delta$, with the
same eigenvalue as that of $f$. Therefore $f-f^\tau$ is again an
eigenvector of $\Delta$. Proposition \ref{transitivity} implies that
$f-f^{\tau}=0$ on $G'$. By Proposition \ref{mp5}, we finally
conclude that
\[
  f-f^\tau=0.
\]

\vsp In the rest of the section, we sketch the proof of Theorem
\ref{thm:mainB} for the case $\BD =\H$ (the real quaternion division
algebra), which is much simpler than the split case of $\GL_6(\K)$.
As in the split case, define a parabolic subgroup $P_\H$ containing
$S_\H$ and the rank matrix $\oR(x)$ (for $x\in G_\H$) in the obvious
way. Then $\oR(x)$ takes the following 6 possible values:
\[
  \left[\begin{array}{cc}
    2&1\\
    1&1\\
  \end{array}\right]=R_{\open},
   \left[\begin{array}{cc}
    2&1\\
    1&0\\
  \end{array}\right],
\]
\[
  \left[\begin{array}{cc}
    1&1\\
    1&1\\
  \end{array}\right],
  \left[\begin{array}{cc}
    1&1\\
    0&0\\
  \end{array}\right],
  \left[\begin{array}{cc}
    1&0\\
    1&0\\
  \end{array}\right],
   \left[\begin{array}{cc}
    1&0\\
    0&0\\
  \end{array}\right],
\]
which gives rise to 6 $P$-$P^\tau$ double cosets $\{G_{\H,R}\}$.

Let $f$ be as in the theorem. If we replace $\GL_2(\K)$ by
$\H^\times$, the analog of Proposition \ref{m2} still holds. This
will imply that $f-f^{\tau}$ vanishes on $G_{\H,R_{\open}}$. As in
the split case, we define a left invariant vector field $X_{\oleft}$
on $G_\H$ using $x_{\oleft}=\left[\begin{smallmatrix}
                  0&1&0\\
                  0&0&1\\
                  0&0&0\end{smallmatrix}\right]\in \gl_3(\H)$.
 Then as in Section \ref{smallsub}, one
checks that $X_{\oleft}$ is transversal to every double coset
$G_{\H,R}$ for $R\neq R_{\open}$. We conclude as in the split case
that $f-f^{\tau}=0$.

\vsp

\noindent {\bf Remarks}:

\begin{itemize}
\item [(a)] Theorem \ref{thm:mainB} in fact holds without the tempered-ness
condition on $f$. But we shall not prove or exploit this fact.
\item [(b)] We also expect Theorem \ref{thm:mainB} to hold without the
assumption that $f$ is an eigenvector of $\Delta_\BD$.
\end{itemize}

\section{Proof of Theorem \ref{thm:mainA}}
\label{theoremA}

The argument of this section is standard, and it works for a more
general real reductive group $G$.

By a representation of $G$, we mean a continuous linear action of
$G$ on a complete, locally convex, Hausdorff complex topological
vector space. We say that a representation $V$ of $G$ is in the class $\CFH$ if it is Fr\'{e}chet, smooth,
of moderate growth, admissible and $\oZ(\gC)$ finite. Here and as
usual, $\oZ(\gC)$ is the center of the universal enveloping algebra
$\oU(\gC)$ of the complexification $\gC$ of $\g$. The reader may
consult \cite{Cass,W2} for more details about representations in the class $\CFH$.

Let $V_1$ and $V_2$ be two representations of
$G$ in the class $\CFH$. We say that they are contragredient to each other if there
exists a nondegenerate continuous $G$ invariant bilinear form
\begin{equation*}\label{pv1v2}
   \la\,,\,\ra: V_1\times V_2\rightarrow \C.
\end{equation*}
If $V_1$ and $V_2$ are  contragredient to each other, then $V_1$ is
irreducible if and only if $V_2$ is.

Let $S_1$ and $S_2$ be two closed subgroups of $G$, with continuous
characters (not necessarily unitary)
\[
  \chi_{S_i}: S_i\rightarrow \C^\times,\quad i=1,2.
\]
Let $\tau$ be a continuous anti-automorphism of $G$ (not necessarily
an anti-involution).

The following is a generalization of the usual Gelfand-Kazhdan
criterion. See \cite{SZ08} for a detailed proof. Recall that
$\operatorname{U}(\gC)^G$ is identified with the space of
bi-invariant differential operators on $G$, as usual.

\begin{gelfand}\label{gelfand}
Assume that for every $f\in \con^{-\xi}(G)$ which is an eigenvector
of $\operatorname{U}(\gC)^G$, the conditions
\[
  f(sx)=\chi_{S_1}(s)f(x), \quad s\in S_1,
\]
and
\[
  f(xs)=\chi_{S_2}(s)^{-1}f(x), \quad s\in S_2
\]
imply that
\[
   f(x^\tau)=f(x).
\]
Then for any two irreducible representations
$V_1$ and $V_2$ of $G$ in the class $\CFH$ which are contragredient to each other, one
has that
\begin{equation*}\label{dhom}
  \dim \Hom_{S_1}(V_1, \C_{\chi_{S_1}}) \, \dim \Hom_{S_2}(V_2,\C_{\chi_{S_2}})\leq 1.
\end{equation*}
\end{gelfand}

\vsp Now we finish the proof of Theorem \ref{thm:mainA}. Assume that
$V_1=V$ is an irreducible representation of
$G$ in the class $\CFH$. Define the irreducible representation
$V_2$ of $G$ in the class $\CFH$ as follows. The representation $V_2$ equals to $V$ as a
topological vector space, and the action $\rho_2$ of $G$ on $V_2$ is
given by
\[
  \rho_2(g)v=\rho_1(g^{-\tau})v,\quad g\in G, v\in V,
\]
where $\rho_1$ is the action of $G$ on $V_1$. Using character theory
and the fact that $g$ is always conjugate to $g^\tau$, we conclude
that $V_1$ and $V_2$ are contragredient to each other \cite[Theorem
2.4.2]{ADE}. Now let
\[
  S_1=S,\quad S_2=S^\tau,\quad \chi_{S_1}=\chi_S,
\]
and
\[
  \chi_{S_2}(g)=\chi_S(g^{-\tau}),\quad g\in S_2.
\]
Theorem \ref{thm:mainB} says that the assumption of Proposition
\ref{gelfand} is satisfied, and so
\[
\dim \Hom_{S_1}(V_1, \C_{\chi_{S_1}}) \, \dim
\Hom_{S_2}(V_2,\C_{\chi_{S_2}})\leq 1.
\]
Note that by the identification $V_1=V_2=V$ as well as the explicit
actions, we have
\[
 \Hom_{S_1}(V_1, \C_{\chi_{S_1}})=
\Hom_{S_2}(V_2,\C_{\chi_{S_2}})=\Hom_{S}(V, \C_{\chi_{S}}).
\]
Hence
\[
  \dim \Hom_{S}(V, \C_{\chi_{S}})\leq 1,
\]
and the proof is complete.

\section{Some consequences}
\label{misc}

\subsection{Uniqueness of trilinear forms}
\label{sub-trilinear}

The following theorem is proved in \cite{Loke} (in an exhaustive
approach), and its p-adic analog was proved much earlier in
\cite[Theorem 1.1]{Prasad}.

\begin{Whittaker2}\label{Whittaker2}
Let $V$ be an irreducible representation of
$\GL_{2}(\K)\times \GL_{2}(\K)\times\GL_{2}(\K)$ in the class $\CFH$. Then
\[
   \dim \Hom_{\GL_{2}(\K)} (V,\C_{\chi_{2}})\leq 1.
\]
Here we view $\GL_{2}(\K)$ as the diagonal subgroup of
$\GL_{2}(\K)\times \GL_{2}(\K)\times\GL_{2}(\K)$, and $\chi
_2=\chi_{\K^\times}\circ \det$ is a character of $\GL_{2}(\K)$.
\end{Whittaker2}
\begin{proof} By the Gelfand-Kazhdan criterion, one just needs to show
the following: let $\GL_2(\K)\times \GL_2(\K)$ act on
\begin{equation*} \label{dm2}
  G_{2,2,2}=\GL_2(\K)\times \GL_2(\K)\times \GL_2(\K)
\end{equation*}
by
\[
  (g_1,g_2)(x,y,z)=(g_1 x g_2^\tau,g_1 y g_2^\tau, g_1 z g_2^\tau),\quad
  g_1,g_2\in \GL_2(\K).
\]
Denote by $\chi_{2,2}$ the character of $\GL_2(\K)\times \GL_2(\K)$
given by
\[
  \chi_{2,2}(g_1,g_2)=\chi_{\K^\times}(\det(g_1))\chi_{\K^\times}(\det(g_2)),
  \quad g_1,g_2\in \GL_2(\K).
  \]
Then for all $f\in \con^{-\xi}_{\chi_{2,2}}(G_{2,2,2})$, we have
\[f(x^{\tau},y^{\tau},z^{\tau})=f(x,y,z).
\]
To show the above, we observe that
$M_2=\GL_2(\K)\times\GL_2(\K)\times \{I_2\}$ is a $\GL_2(\K)\times
\GL_2(\K)$ slice of $G_{2,2,2}$, which is stable under $H_2=\{(x,
x^{-\tau})\mid x\in \GL_2(\K)\}\subset \GL_2(\K)\times \GL_2(\K)$
and $\tau$. The result then follows from Proposition \ref{m2}.
\end{proof}

As noted near the end of Section \ref{theoremB}, if we replace
$\GL_2(\K)$ by $\H^\times$, the analog of Proposition \ref{m2} still
holds (again by using Proposition \ref{first-occur}). Thus the
analog of Theorem \ref{Whittaker2} for $\H^\times$ holds. Of course
this is well-known and easier.

\subsection{Uniqueness of the Jacquet-Shalika model for
$\GL_3(\K)$}\label{sub-gl3} \vsp

Let $L_3$ and $N_3$ be the subgroups of $\GL_3(\K)$, as in Section
\ref{sgl3}. Write $S_3=L_3 N_3$, and
\[
   \chi_{S_3} \left(\,
   \left[\begin{array}{ccc} 1&0&d\\ 0&1&c\\0&0&1\\
                   \end{array}\right]
                   \left[\begin{array}{ccc} a&0&0\\ 0&1&0\\0&0&a\\
                   \end{array}\right]
                   \right)=\chi_{\K^\times}(a)\psi_{\K}(d),
\]
which defines a character of $S_3$.

\begin{Whittaker3}\label{Whittaker3}
Let $V$ be an irreducible representation of
$\GL_{3}(\K)$ in the class $\CFH$. Then
\[
   \dim \Hom_{S_3} (V,\C_{\chi_{S_3}})\leq 1.
\]
\end{Whittaker3}
\begin{proof} As a corollary of Proposition \ref{m32}, we know that if $f\in
\con^{-\xi}(\GL_{3}(\K))$ satisfies
\[
  f(sx)=f(xs^\tau)=\chi_{S_3}(s)f(x),\,\, \textrm{for all $s\in
  S_3$},
\]
then
\[
  f(x^\tau)=f(x).
\]
The theorem then follows, as in Section \ref{theoremA}.
\end{proof}

We remark that the p-adic analog of Theorem \ref{Whittaker3} holds
true, as the same proof goes through.

\vsp \noindent {\bf Remark}: By inducing the character
$\chi_{S_3}$ to a Heisenberg group, one may obtain uniqueness of
the Fourier-Jacobi model for $\GL_3(\K)$.

\subsection{Uniqueness of a certain model for $\GL_4(\K)\times
\GL_2(\K)$} \label{sub-gl4} \vsp

Recall from the Introduction:
\[
   S_{4,2}=(\GL_4(\K)\times \GL_2(\K))\cap S =\left\{\,\left[\begin{array}{ccc} a&b&0\\ 0&a&0\\0&0&a\\
                   \end{array}\right]\in G\right\},
\]
and $\chi_{S_{4,2}}=\chi _S|_{S_{4,2}}$.

\begin{Whittaker4}\label{Whittaker4}
Let $V$ be an irreducible representation of
$\GL_{4}(\K)\times \GL_2(\K)$ in the class $\CFH$. Then
\[
   \dim \Hom_{S_{4,2}} (V, \C_{\chi_{S_{4,2}}})\leq 1.
\]
\end{Whittaker4}
\begin{proof}
Denote by $\Delta_{4,2}$ the Casimir operator on $\GL_4(\K)\times
\GL_2(\K)$ associated to the real trace form. Arguing as in Section
\ref{theoremA}, we will just need to show that, if $f\in
\con^{-\xi}(\GL_4(\K)\times \GL_2(\K))$ is an eigenvector of
$\Delta_{4,2}$, and if
\[
  f(sx)=f(xs^\tau)=\chi_{S_{4,2}}(s)f(x),\,\, \textrm{for all $s\in
  S_{4,2}$},
\]
then
\[
  f(x^\tau)=f(x).
\]

To conclude the above, we further assume that $f(x^\tau)=-f(x)$. We
need to show that $f=0$.

Denote
\[
 C_{4,2}=\left\{\,\left[\begin{array}{ccc} a_{11}&a_{12}&0\\
    a_{21}&a_{22}&0\\
     0&0&y
                   \end{array}\right]\in \GL_4(\K)\times \GL_2(\K)\mid
   \textrm{$y^{-1}a_{22}$ is nilpotent}\right\}.
\]
This is the union of $Z_4$ (in Section \ref{sm4}) and $Z_4'$, where
\[
   Z'_4=\left\{\,\left[\begin{array}{ccc} a_{11}&a_{12}&0\\
       a_{21}&0&0\\
       0&0&y\\
      \end{array}
          \right]\in \GL_4(\K)\times \GL_2(\K)\,\right\}.
\]
By using Proposition \ref{m2} and Proposition \ref{m3}, we first
show that $f$ is supported in $C_{4,2}$. Proposition \ref{m4}
further implies that $f$ can only be supported in $Z_4'$.

Now set
\[
  x_{4,\oleft}=\left[
               \begin{array}{ccc}
                  0&I_2&0\\
                  0&0&0\\
                  0&0&0\\
                \end{array}
            \right]\in \gl_4(\K)\times \gl_2(\K),
\]
and denote by $X_{4,\oleft}$ the left invariant vector field on
$\GL_4(\K)\times \GL_2(\K)$ whose tangent vector at $x\in G$ is
$xx_{4,\oleft}$. As in Section \ref{smallsub}, one checks that
$X_{4,\oleft}$ is transversal to $Z_4'$. We may then conclude that
$f=0$, as in Section \ref{theoremB}.
\end{proof}

\subsection{Uniqueness of Whittaker models} \label{sub-whittaker}

Let $\mathbf{G}$ be a quasisplit connected reductive algebraic group
defined over $\R$. Let $\mathbf{B}$ be a Borel subgroup of
$\mathbf{G}$, with unipotent radical $\mathbf{N}$. Let
\[
  \chi_\mathbf{N}:\mathbf{N}(\R)\rightarrow \C^\times
\]
be a generic unitary character. The meaning of ``generic" will be
explained later in the proof.

The following theorem is fundamental and well-known. For
$\mathbf{G}=\GL_n$, this is a celebrated result of Shalika
\cite{Shalika}. A proof in general may be found in \cite[Theorem
9.2]{CHM}. We shall give a short proof based on the notion of
unipotent $\chi$-incompatibility.

\begin{Whittaker}
Let $V$ be an irreducible representation of
$\mathbf{G}(\R)$ in the class $\CFH$. Then
\[
   \dim \Hom_{\mathbf{N}(\R)} (V,\C_{\chi_{\mathbf{N}}})\leq 1.
\]
\end{Whittaker}

\begin{proof} We say that a representation is in the class $\CDH$ if it is the
strong dual of a representation in the class $\CFH$. The current
theorem can then be reformulated as follows: the space
\[
 U^{\chi_\mathbf{N}^{-1}}=\{u\in U\mid gu=
 \chi_\mathbf{N}^{-1}(g)u\,\textrm{ for all } g\in \mathbf{N}(\R)\}
\]
is at most one dimensional for every irreducible representation $U$ of $\mathbf{G}(\R)$ in the class $\CDH$.

Let $\bar{\mathbf{B}}$ be a Borel subgroup opposite to $\mathbf{B}$,
with unipotent radical $\bar{\mathbf{N}}$. Then
$\mathbf{T}=\mathbf{B}\cap \bar{\mathbf{B}}$ is a maximal torus. Let
\[
  \chi_\mathbf{T}:\mathbf{T}(\R)\rightarrow \C^\times
\]
be an arbitrary character. Then
\[
\begin{aligned}
  U(\chi_\mathbf{T})=&\{f\in\con^{-\infty}(\mathbf{G}(\R))\mid f(t\bar{n}x)
  =\chi_\mathbf{T}(t)f(x)\, \\
     & \textrm{ for all } t\in \mathbf{T}(\R), \bar{n}\in \bar{\mathbf{N}}(\R)\}
\end{aligned}
\]
is the distributional version of nonunitary principal series
representations. By Casselman's subrepresentation theorem (in the
category of representations in the class $\CDH$), it
suffices to show that
\begin{equation}\label{dimu}
  \dim U(\chi_\mathbf{T})^{\chi_\mathbf{N}^{-1}}\leq 1, \quad
  \text{for any $\chi_\mathbf{T}$}.
\end{equation}

Let
\[
   H_\mathbf{G}=\bar{\mathbf{B}}(\R)\times \mathbf{N}(\R),
\]
which acts on $\mathbf{G}(\R)$ by
\[
   (\bar{b}, n)x=\bar{b}xn^{-1}.
\]
Write
\[
  \chi_\mathbf{G}(t\bar{n},n)=\chi_\mathbf{T}(t) \chi_\mathbf{N}(n),
\]
which defines a character of $H_\mathbf{G}$. Then (\ref{dimu}) is
equivalent to
\begin{equation}\label{dimcong}
   \dim \con^{-\infty}_{\chi_\mathbf{G}}(\mathbf{G}(\R))\leq 1.
\end{equation}

Let $W$ be the Weyl group of $\mathbf{G}(\R)$ with respect to
$\mathbf T$.
We have the Bruhat decomposition
\[
  \mathbf{G}(\R)=\bigsqcup_{w\in W} \mathbf{G}_w,\quad
  \textrm{with}\quad
  \mathbf{G}_w=\bar{\mathbf{B}}(\R)w \mathbf{N}(\R).
\]
From this we form a $H_\mathbf{G}$ stable filtration
\[
  \emptyset=\mathbf{G}^0\subset
  \mathbf{G}^1\subset
  \mathbf{G}^2\subset\cdots\subset \mathbf{G}^r=\mathbf{G}(\R)
\]
of $\mathbf{G}(\R)$ by open subsets, with
$\mathbf{G}^1=\bar{\mathbf{B}}(\R)\mathbf{N}(\R)$ and every
difference $\mathbf{G}^i\setminus \mathbf{G}^{i-1}$ a Bruhat cell
$\mathbf{G}_w$, for $i\geq 2$.

Clearly (\ref{dimcong}) is implied by the following two assertions:
\begin{equation}\label{dimcong2}
  \dim \con^{-\infty}_{\chi_\mathbf{G}}(\mathbf{G}^1)=1;
\end{equation}
and
\begin{equation}\label{dimcong3}
  \textrm{if $f\in
\con^{-\infty}_{\chi_\mathbf{G}}(\mathbf{G}^{i})$ vanishes on
$\mathbf{G}^{i-1}$, then $f=0$,}
\end{equation}
for $i\geq 2$. The equality (\ref{dimcong2}) is clear as
$\mathbf{G}^1=\bar{\mathbf{B}}(\R)\mathbf{N}(\R)$. For
(\ref{dimcong3}), we write
\[
  \mathbf{G}^i\setminus \mathbf{G}^{i-1}=\mathbf{G}_w, \quad
  \textrm{with $w$ a non-identity element of $W$}.
\]
The genericity means that $\chi_\mathbf{N}$ has nontrivial
restriction to $\mathbf{N}(\R)\cap w^{-1}(\bar{\mathbf{N}}(\R))w$.
Pick
\[
 n=w^{-1}\bar{n} w\in \mathbf{N}(\R)\cap
w^{-1}(\bar{\mathbf{N}}(\R))w
\]
so that $\chi_\mathbf{N}(n)\neq 1$. Then $(\bar{n},n)\in
H_\mathbf{G}$ satisfies
\[
  (\bar{n},n)w=w, \quad \textrm{and  }
  \chi_\mathbf{G}(\bar{n},n)=\chi_{\mathbf{N}}(n)\neq 1.
\]
Consequently, $\mathbf{G}_w$ is unipotently
$\chi_\mathbf{G}$-incompatible. Now (\ref{dimcong3}) follows from
Lemma \ref{localization3}.
\end{proof}

\end{document}